\documentclass[12pt]{elsarticle}

\usepackage{amsmath,amsthm}
\usepackage{amssymb}
\usepackage{xcolor}

\numberwithin{equation}{section}
\theoremstyle{plain}
\newtheorem{thm}{Theorem}[section]
\newtheorem{prop}[thm]{Proposition}
\newtheorem{cor}[thm]{Corollary}

\newtheorem{df}[thm]{Definition}
\newtheorem{ex}[thm]{Example}
\newtheorem{rem}[thm]{Remark}

\newcommand{\ds}{\displaystyle}

\newcommand{\Hp}{\ensuremath{H^2_+}}
\newcommand{\Hpn}{\ensuremath{\left(\Hp\right)^n}}
\newcommand{\Hm}{\ensuremath{H^2_-}}

\newcommand{\Hpm}{\ensuremath{H^2_\pm}}
\newcommand{\Hpmn}{\ensuremath{\left(\Hpm\right)^n}}
\newcommand{\Linf}{\ensuremath{L^\infty}}
\newcommand{\Linfn}{\ensuremath{\left(\Linf\right)^{n\times n}}}
\newcommand{\Hinf}{\ensuremath{H^\infty}}
\newcommand{\TT}{\ensuremath{\mathbb{T}}}
\newcommand{\DD}{\ensuremath{\mathbb{D}}}
\newcommand{\CC}{\ensuremath{\mathbb{C}}}
\newcommand{\CCp}{\ensuremath{\mathbb{C}^+}}
\newcommand{\CCm}{\ensuremath{\mathbb{C}^-}}
\newcommand{\ZZ}{\ensuremath{\mathbb{Z}}}
\newcommand{\NNN}{\ensuremath{\mathbb{N}}}
\renewcommand{\AA}{\ensuremath{\mathcal{A}}}
\newcommand{\BB}{\ensuremath{\mathcal{B}}}
\newcommand{\BBo}{\ensuremath{{\mathcal{B}_0}}}
\newcommand{\CCC}{\ensuremath{\mathcal{C}}}
\newcommand{\OO}{\ensuremath{\mathcal{O}}}
\newcommand{\GG}{\ensuremath{\mathcal{G}}}
\newcommand{\HH}{\ensuremath{\mathcal{H}}}
\newcommand{\RR}{\ensuremath{\mathcal{R}}}
\newcommand{\UU}{\ensuremath{\mathcal{U}}}
\newcommand{\KK}{\ensuremath{\mathcal{K}}}
\newcommand{\KKmin}{\ensuremath{\mathcal{K}_{min}}}
\newcommand{\NN}{\ensuremath{\mathcal{N}^+}}
\newcommand{\Cg}{\ensuremath{C_{\ol g}}}

\newcommand{\ol}[1]{\ensuremath{\overline{#1}}}
\newcommand{\annot}[2]{\underbrace{#1}_{\smash[b]{\scriptstyle #2}}}
\newcommand{\bs}{\backslash}
\newcommand{\zero}{\ensuremath{\{0\}}}
\newcommand{\spn}{\ensuremath{\operatorname{span}}}

\renewcommand{\Re}{\ensuremath{\operatorname{Re}}}

\newcommand{\kktil}{\ensuremath{\tilde k^\theta_\lambda}}
\newcommand{\kk}{\ensuremath{k^\theta_\lambda}}
\newcommand{\kktilo}{\ensuremath{\tilde k^\theta_0}}
\newcommand{\kko}{\ensuremath{k^\theta_0}}
\newcommand{\gnotzero}{\ensuremath{g\in\Linf\bs\zero}}

\renewcommand{\subset}{\subseteq}
\renewcommand{\supset}{\supseteq}

\begin{document}

\begin{frontmatter}

\title{Maximal functions, conjugations and multipliers between Toeplitz kernels}

\author[a]{M. Cristina C\^{a}mara\corref{cor1}}
\ead{cristina.camara@tecnico.ulisboa.pt}

\author[a]{C. Carteiro}
\ead{carlos.carteiro@tecnico.ulisboa.pt}

\author[a,b]{C. Diogo}
\ead{cristina.diogo@iscte-iul.pt}

\affiliation[a]{
	organization={Center for Mathematical Analysis, Geometry and Dynamical Systems, Department of Mathematics, Instituto Superior Tecnico, Universidade de Lisboa},
	addressline={Av. Rovisco Pais},
	postcode={1049-001},
	city={Lisboa},
	country={Portugal}
}

\affiliation[b]{
	organization={ISCTE - Lisbon University Institute},
	addressline={Av. das Forças Armadas},
	postcode={1649-026},
	city={Lisboa},
	country={Portugal}
}

\cortext[cor1]{Corresponding author}

\begin{abstract}
Toeplitz kernels can be defined by Riemann-Hilbert problems, by maximal functions, or by multipliers acting on model spaces.
In this paper we study those different characterisations and their relations, highlighting, on the one hand, the crucial role played by symbol factorisation in obtaining multipliers from a model space onto a Toeplitz kernel, in particular isometric multipliers, and, on the other hand, a deep connection of maximal functions with a naturally defined conjugation on the Toeplitz kernel.
\end{abstract}


\end{frontmatter}

\section{Introduction: Sarason's question}
\label{sec:intro}

When is a closed subspace of $H^2(\DD)$ a Toeplitz kernel, i.e., the kernel of a bounded Toeplitz operator? To this question, asked by Sarason in \cite{Sa94} where he revisits and reformulates the results of Hayashi in \cite{Ha85,Ha86,Ha90}, various answers can be given by looking at it from different perspectives which complement and clarify one another.

\subsection{Toeplitz kernels and Riemann-Hilbert problems}
\label{subsec:TK as RH problems}

Let \TT\ be the unit circle, $\Linf := \Linf(\TT)$, $\Hp:=H^2(\DD)$ denote the Hardy space of the unit disk and let $\Hm := \left(\Hp\right)^\bot = \bar z \ol \Hp$.
For $G \in \Linfn$, consider the Riemann-Hilbert (RH for short) boundary value problem on \TT\ with matrix coefficient $G$
\begin{equation}
\label{eq:RH nxn}
G \phi_+ = \phi_- , \quad  \phi_\pm \in \Hpmn.
\end{equation}

The solutions $\phi_+$ to the RH problem \eqref{eq:RH nxn} constitute the kernel of the Toeplitz operator with symbol $G$, defined by
\begin{equation}
\label{def:Toeplitz operator}
T_G: \Hpn \to \Hpn, \quad T_G \phi_+ = P^+ G \phi_+,
\end{equation}
where $P^+$ denotes the orthogonal projection from $L^2$ onto \Hp, applied componentwise to $G\phi_+$.
We have thus the following characterisation:

\begin{enumerate}
\item[(i)] a closed subspace $\KK\subset\Hpn$ is a Toeplitz kernel if and only if there exists $G\in\Linfn$ such that $\KK$ coincides with the space of solutions $\phi_+$ of the RH problem \eqref{eq:RH nxn}.
In that case $\KK=\ker T_G$.
\end{enumerate} 

This characterisation, which is valid both in the scalar and in the matricial cases and can moreover be extended to Toeplitz operators defined in the Hardy spaces $H^p\left(\DD\right)$, $1<p<\infty$, will be frequently used in this work.
Not only is it very general, as it also allows for purely complex analytic methods to be applied when studying Toeplitz operators.
Model spaces in \Hp, which are a particularly important class of Toeplitz kernels  (\cite{GarMashRoss16}), can be seen as consisting of the solutions $\phi_+\in\Hp$ of the particular RH problem 
\begin{equation*}
\ol \theta \phi_+ = \phi_-, \quad \phi_\pm \in \Hpm,
\end{equation*}
where $\theta$ is an inner function.

Many important properties of Toeplitz kernels can be obtained by taking this RH approach.
One such property is that Toeplitz kernels in \Hpn\ are {\em nearly $\eta$-invariant}, i.e.,
\begin{equation}
\label{eq:nearly_invariance}
\phi_+ \in \ker T_G, \; \eta\phi_+\in\Hpn \implies \eta\phi_+ \in \ker T_G,
\end{equation}
for $\eta$ in a wide class of complex valued functions, including all functions in $\ol\Hinf$ and all rational functions whose poles are in the closed unit disk (\cite{CP14}).

As a consequence of this property we can also characterise some subspaces of \Hpn\ which are {\em not} Toeplitz kernels.

\begin{prop}[{\cite[Thms. 3.5, 3.6]{CP14}}]
\label{prop:subspaces not TK}
If $\AA \subset f \cdot \Hpn$, where $f$ is an inner function or a rational function whose zeroes belong to the closed unit disk, then $\AA$ is \emph{not} a Toeplitz kernel.
\end{prop}

In particular, for $\eta = 1/z$ (i.e., $\eta = \ol z$ on \TT), \eqref{eq:nearly_invariance} is equivalent to
\begin{equation*}
\phi_+ \in \ker T_G, \phi_+(0) = 0 \implies S^*\phi_+ \in \ker T_G,
\end{equation*}
where $S^* = T_{\ol z I_{n\times n}}$ is the backward shift operator.
This is expressed by saying that Toeplitz kernels are {\em nearly $S^*$-invariant subspaces} of \Hpn.
In the context of \Hp, nearly $S^*$-invariant subspaces were introduced and characterised by D. Hitt in \cite{Hitt88}.
This study has since attracted great interest and has become a topic of research in its own right (see for instance \cite{ChaPar11} and references therein).

\subsection{Toeplitz kernels and maximal functions}
\label{subsec:TK maximal functions}

We now focus on Toeplitz operators defined on \Hp, with scalar symbols $g \in \Linf$.

Another important property of Toeplitz kernels that can be deduced from the RH description of those spaces is that, for every $f\in\Hp$, there exists a Toeplitz kernel containing $\{ f \}$, which is contained in every other Toeplitz kernel to which $f$ belongs.
It is called the \emph{minimal Toeplitz kernel for $f$} and denoted by $\KKmin (f)$; $f$ is said to be a \emph{maximal function} in $\KKmin (f)$.
Moreover, for every non-zero Toeplitz kernel \KK, one can find some maximal function $f^M$ and, if $f^M = IO$ is an inner-outer factorisation of $f^M$ (with $I$ inner and $O\in\Hp$ outer), we have that
\begin{equation}
\label{eq:f=IO.max.gives.unimodular.symbol}
\KK = \KKmin (f^M) = \ker T_{\ol{zI}\frac{\ol {O}}{O}}
\end{equation}
(\cite[Thm 5.1]{CP14}).
Thus any maximal function in a Toeplitz kernel determines the latter and provides, by inner-outer factorisation, a unimodular symbol for that kernel.

We have thus the following characterisation of Toeplitz kernels in terms of maximal functions.
\begin{enumerate}
\item[(ii)] a closed subspace $\KK \subset \Hp$ is a Toeplitz kernel if and only if there exists $f^M \in \KK$ with inner-outer factorisation $f^M = IO$ ($I$ inner and $O$ outer) such that, for all $\phi_+ \in \Hp$,
\begin{equation*}
\phi_+ \in \KK \iff \frac{\ol{f^M}}{O}\, \phi_+ \in \ol\Hp.
\end{equation*}
In that case, $\KK = \ker T_{\ol{z} \frac{\ol{f^M}}{O}} = \ker T_{\ol{zI}\frac{\ol {O}}{O}}$.
\end{enumerate}

Maximal functions in a Toeplitz kernel can be described as follows.

\begin{thm}[{\cite[Thm 2.2]{CP18_multipliers}}]
\label{thm: maximal function characterization}
Let $g\in \Linf\bs\{0\}$.
Then $f^M \in \Hp$ is a maximal  function in $\ker T_g$ if and only if 
\begin{equation}
\label{eq:maximal_function_characterization}
g f^M = \ol z \ol O , \quad \text{ for some } O \in \Hp \text{, outer}.
\end{equation}
\end{thm}

\begin{ex} \label{ex:ktil_is_max_func}
Let $g=\ol \theta$, where $\theta$ is an inner function, and define, for any $\lambda\in\DD$,
\begin{equation}
\label{eq:def_k_k^til}
\kktil = \frac{\theta - \theta(\lambda)}{z-\lambda} \quad , \quad \kk = \frac{1-\ol{\theta(\lambda)}\theta}{1-\ol\lambda z}.
\end{equation}

Note that $\kktil\in\Hinf$ and $\kk \in \GG\Hinf$ (denoting by $\GG\AA$ the group of invertible elements in a unital algebra \AA), so \kk\ is outer and we have
\begin{equation*}
\ol\theta \kktil = \ol z \ol\kk.
\end{equation*}

Thus, \kktil\ is a maximal function in $K_\theta$ for any $\lambda\in\DD$.
\end{ex}

\subsection{Toeplitz kernels and multipliers acting on model spaces}
\label{subsec: Hayashi's representation of Toeplitz kernels}

Toeplitz kernels can also be described as resulting from multipliers acting on model spaces.
Multipliers between two Toeplitz kernels play an important role in understanding those spaces, especially if one of the Toeplitz kernels is a simple or a well studied space to which the other Toeplitz kernel is reduced, in some sense, by a multiplication operator.
The case where one of the Toeplitz kernels is a model space and we have a representation of the form
\begin{equation}
\label{eq:Hayashi_representation}
\ker T_g = w K_\theta , \quad \theta \text{ inner},
\end{equation}
is particularly important.
Indeed Hayashi showed (\cite[Thm. 3]{Ha86}, see also \cite{Sa94}) that such a representation is possible for every nontrivial Toeplitz kernel and that one may choose the multiplier $w$ and the function $\theta$ in \eqref{eq:Hayashi_representation} such that $\theta (0) = 0$ and $w$ multiplies the model space $K_\theta$ isometrically onto $\ker T_g$.

Hayashi's characterisation of Toeplitz kernels can be stated as follows.
Let $u\in\Hp$ be an outer function with $u(0)>0$, let $F$ be the Herglotz integral of $|u|^2$, let $b= \frac{F-1}{F+1}$, $a= \frac{2u}{F+1}$
and write, for $\alpha$ inner, $u_\alpha = \frac{a}{1-\alpha b}$.
With these notations we have:

\begin{thm}[{\cite[Theorem 5]{Ha90}}]
\label{thm:Hayashi_rep}
The nontrivial kernels of Toeplitz operators are the subspaces of \Hp\ of the form
\begin{equation}
\label{eq:Hayashi_rep_v2}
M = u_\alpha K_{z \alpha}
\end{equation}
where $\alpha$ is inner, $u\in\Hp$ is outer with $u(0)>0$, $u^2$ is an exposed point of the unit ball of $H^1$, and $u_\alpha$ multiplies $K_{z \alpha}$ isometrically onto $M$.
\end{thm}

A function $f$ in the unit ball of $H^1$ is exposed if and only if it is rigid and $||f||_1 = 1$.
Rigid functions can be characterised in terms of Toeplitz operators (\cite[Proposition 1]{Sa94}) as follows.

\begin{prop}
\label{prop:square.rigid}
If $O\in\Hp$ is outer, then 
\[
O^2 \text{ is rigid in } H^1 \iff \ker T_\frac{\ol O}{O} = \{ 0 \} \iff \ker T_{\ol z \frac{\ol O}{O}} = \spn \{ O \}.
\]
\end{prop}

When $f \in \Hp$ is such that $f^2$ is rigid in $H^1$, we say that $f$ is {\em square-rigid} (\cite{CP16_finite}).
In that case, it is easy to see that $f$ is necessarily outer (\cite{Sa94}).

As a consequence of Proposition \ref{prop:square.rigid}, we have:

\begin{cor}
If $f^{\pm 1} \in \Hp$ then $f$ is square-rigid.
\end{cor}

Hayashi's representation \eqref{eq:Hayashi_rep_v2} is unique up to a unimodular constant.
While this representation is important to understand the structure of a Toeplitz kernel, by describing it in terms of a model space and a multiplier that preserves the norm, it may be very difficult to obtain explicitly.
In general, we look for a description of the form \eqref{eq:Hayashi_representation}, which we call a {\em model space representation} of $\ker T_g$. 
The multiplier in \eqref{eq:Hayashi_representation} is not necessarily isometric, but \eqref{eq:Hayashi_representation} provides a clear description of the Toeplitz kernel, especially if the model space is well understood.
We say that \eqref{eq:Hayashi_representation} is an {\em isometric model space representation} of $\ker T_g$ if $w$ multiplies $K_\theta$ isometrically onto $\ker T_g$.

\vspace{8pt}

In this paper we study these various characterisations of a Toeplitz kernel and their relations, highlighting the crucial role played in this study by appropriate symbol factorisations and by conjugations.

The paper is organised as follows.

In Section \ref{sec:Maximal_functions_and_multipliers_between_Toeplitz_kernels}, which is of a preliminary nature, we present several known and new results highlighting the close connection between multipliers from one Toeplitz kernel onto another, maximal functions, and associated factorisations for the symbol that naturally arise in this context.

In Section \ref{sec:Toeplitz kernels and symbol factorisation}, we review and extend the characterisation of several properties of Toeplitz operators and their kernels in terms of certain factorisations of their symbols.
By introducing a generalisation of the classical notion of $L^2$-factorisation, we obtain a model space representation for the kernels of Toeplitz operators with piecewise continuous symbols, from which Hayashi's representation \eqref{eq:Hayashi_rep_v2} can be determined.
We also introduce, from the characterisation of Toeplitz kernels in terms of maximal functions, a new notion of maximal function factorisation for the symbols and we establish conditions for the outer factor of a maximal function to be a multiplier from a model space onto the kernel.
We illustrate these results by obtaining two model space representations of $\ker T_{\ol E z}$, where $E(z) = \exp (\frac{z+1}{z-1} )$, one of which is isometric.

In Section \ref{sec: Maximal functions and conjugations} we study, for the first time to the authors' knowledge, the relations between maximal functions and the natural conjugation in $\ker T_g$, denoted \Cg\ ($|g|=1$), using it to establish connections with generators of model spaces and invariant sets for the conjugation.

In Section \ref{sec:Applications} we apply the results of the previous sections to model spaces, a particularly well studied class of Toeplitz kernels, as an illustration.
In particular, by using maximal function factorisations, we obtain a large class of multipliers between two model spaces, thus recovering and extending results from \cite{Cro94}.

%%%%%%%%%%%%%                  Section 2                 %%%%%%%%%%%%%
%%%%% Maximal functions and multipliers between Toeplitz kernels %%%%%

\section{Maximal functions and multipliers between Toeplitz kernels}
\label{sec:Maximal_functions_and_multipliers_between_Toeplitz_kernels}

In what follows, $g$ and $h$ always denote functions in $\Linf\bs\zero$ such that $\ker T_g, \ker T_h \neq\zero$.

There is a close connection between maximal functions in Toeplitz kernels and multipliers between Toeplitz kernels.
We start by noting that, while multipliers between model spaces must be in \Hp\ (\cite[Corollary 4]{Cro94}), multipliers between general Toeplitz kernels need not be in \Hp; they must however belong to the Smirnov class \NN\ (\cite{Ni02_vol1}).
So we have that
\begin{align*}
w \ker T_h \subset \ker T_g & \implies w \in \NN \\
w \ker T_h = \ker T_g & \implies w^{\pm 1} \in \NN
\end{align*}
(\cite[Remark 2.4]{CP18_multipliers}).
These multipliers are characterised as follows.

\begin{thm}[{\cite[Thm. 2.5]{CP18_multipliers}}]
\label{thm:multiplicadores_into_kernels}
Let $w\in\NN$.
Then the following are equivalent:
\begin{enumerate}[(i)]
\item $w \ker T_h \subset \ker T_g$;

\item $w \ker T_h \subset L^2$ and $\ds w g/h \in \ol\NN$;

\item $w \ker T_h \subset L^2$ and for some (and hence every) maximal function $\ds f^M_h$ in $\ker T_h$ we have $\ds w f^M_h \in \ker T_g$.
\end{enumerate}

\end{thm}

For $w=1$ we obtain from Theorem \ref{thm:multiplicadores_into_kernels} necessary and sufficient conditions for a Toeplitz kernel to be included in another Toeplitz kernel.

\begin{cor}[{\cite[Prop. 2.16]{CP18_multipliers}}]
\label{cor:kernel.subset.of.another.kernel}
The following conditions are equivalent:
\begin{enumerate}[(i)]
\item $\ker T_h \subset \ker T_g$;

\item $g/h \in \ol\NN$;

\item There exists a maximal function $f^M_h$ in $\ker T_h$ such that $f^M_h \in \ker T_g$.
\end{enumerate}

\end{cor}

Since every Toeplitz kernel can be represented in the form \eqref{eq:Hayashi_representation}, it is natural to ask when is a product of the form $w K_\theta$ a Toeplitz kernel and, more generally, when is $w \ker T_h$ a Toeplitz kernel.
We have the following.

\begin{thm}[{\cite{CP18_multipliers}}]
\label{thm:multiplicadores_onto_kernels}
Let $w^{\pm 1} \in \NN$.
Then the following are equivalent:
\begin{enumerate}[(i)]
\item $w \ker T_h = \ker T_g$;

\item $w \ker T_h \subset L^2$, $w^{-1} \ker T_g \subset L^2$ and $\frac{g}{h} = \frac{\ol w}{w} \frac{\ol O_1}{\ol O_2}$ with $O_1,O_2 \in \Hp$, outer;

\item $w \ker T_h \subset L^2$, $w^{-1} \ker T_g \subset L^2$ and, for some (and hence every) maximal function $f^M_h$ in $\ker T_h$, we have that $w f^M_h$ is a maximal function in $\ker T_g$.
\end{enumerate}

Moreover,
\begin{equation}
\label{eq:simbolo_vindo_de_multiplicador}
\ker T_g = w \ker T_h \implies \ker T_g = \ker T_{h\frac{\ol w}{w}}.
\end{equation}

\end{thm}

For $w=1$, we obtain from Theorem \ref{thm:multiplicadores_onto_kernels} necessary and  sufficient conditions for two Toeplitz operators with different symbols to have the same kernel.

\begin{cor}[{\cite[Corollary 2.19]{CP18_multipliers}}]
\label{cor:kernels_iguais}
$\ker T_g = \ker T_h \iff \frac{g}{h} = \frac{\ol O_1}{\ol O_2}$ with $O_1,O_2\in\Hp$ outer.

In this case, if $h\in\GG\Linf$ then $g = h \ol{O_+}$, with $O_+\in\Hinf$, outer; in particular, if $|g| = |h| = 1$ then $g=\lambda h$ with $\lambda\in\CC$, $|\lambda| =1$.

\end{cor}

\begin{cor}[{\cite[Corollary 6.3]{CP24}}]
\label{cor:symbol_zero_on_T}
Let $a\in\TT$, $g\in\Linf$. Then
\[
\ker T_{(z-a)g} = \ker T_{zg}.
\]
\end{cor}

If $\ker T_h$ is a model space in the previous theorem, then we can say more about the multiplier $w$: it must be in \Hp\ and it must be square-rigid.
This is a consequence of the following propositions.

\begin{prop}
\label{prop:multipliers.are.square-rigid}
Let $O\in\Hp$ be outer.
If $O^{-1} \ker T_{\ol z \frac{\ol O}{O}} \subset L^2$ then $O$ is square-rigid.
\end{prop}

\begin{proof}
Let $f_+\in \ker T_{\ol z \frac{\ol O}{O}}$, i.e., $\ol z \frac{\ol O}{O} f_+ = f_-$ with $f_\pm \in \Hpm$.
If $O^{-1} \ker T_{\ol z \frac{\ol O}{O}} \subset L^2$, then $O^{-1} f_+ \in \NN \cap L^2 = \Hp$ and
\[
\annot{O^{-1} f_+}{\in \Hp}
 = \annot{z f_-\ol O^{-1}}{\ol\NN \cap L^2}
   \in \Hp \cap \ol\Hp = \CC.
\]
Therefore, $f_+ = a O$ with $a\in\CC$ and it follows that $\ker T_{\ol z \frac{\ol O}{O}} = \spn \{ O \}$, i.e., $O$ is square-rigid.
\end{proof}

\begin{prop}
\label{prop:multipliers.are.square-rigid.2}
Let $O\in\Hp$ be outer.
If there exists a function $f\in\ker T_h$ such that $f\in\GG\Hinf$, and $O \ker T_h$ is a Toeplitz kernel, then $O$ is square-rigid.
\end{prop}

\begin{proof}
Let $O \ker T_h = \ker T_g$.
We have that $Of\in\ker T_g$, so $\KKmin(Of) = \ker T_{\ol z \frac{\ol{Of}}{Of}} \subset \ker T_g$.
On the other hand, $O^{-1} \ker T_g \subset \Hp$, so $f^{-1} O^{-1} \ker T_g \subset \Hp$.
Therefore
\[
f^{-1} O^{-1} \ker T_{\ol z \frac{\ol{Of}}{Of}} \subset f^{-1} O^{-1} \ker T_g \subset \Hp
\]
and, by Proposition \ref{prop:multipliers.are.square-rigid}, $Of$ is square-rigid.
Since $f\in\GG\Hinf$, it follows that $O$ is square-rigid.
\end{proof}

\begin{cor}
\label{cor:mult.from.model.space.onto.kernel.are.square-rigid}
If $\ker T_g = w K_\theta$ then $w\in\Hp$ and $w$ is square-rigid.
\end{cor}

\begin{proof}
Since $f=\kk \in K_\theta \cap \GG\Hinf$, the result follows from Proposition \ref{prop:multipliers.are.square-rigid.2}.
\end{proof}

A close connection between multipliers from one Toeplitz kernel into another and maximal functions in those Toeplitz kernels becomes clear from the previous results:
on the one hand, maximal functions completely determine a Toeplitz kernel and can be used as test functions for multipliers;
on the other hand, multipliers from $\ker T_h$ onto $\ker T_g$ determine, by multiplication, maximal functions in $\ker T_g$ from maximal functions in $\ker T_h$.

It also becomes clear moreover that, whether we look at a Toeplitz kernel as defined by a maximal function, or we look at a representation of the form \eqref{eq:Hayashi_representation}, certain factorisations of the symbol naturally appear, either via \eqref{eq:f=IO.max.gives.unimodular.symbol} or via \eqref{eq:simbolo_vindo_de_multiplicador}.
In the following section we study different kinds of symbol factorisations for Toeplitz operators and their relations with maximal functions and multipliers.

%%%%%%%%%%%%%           Section 3         %%%%%%%%%%%%%
%%%%%  Toeplitz kernels and symbol factorisation  %%%%%

\section{Toeplitz kernels and symbol factorisation}
\label{sec:Toeplitz kernels and symbol factorisation}

The study of Toeplitz operators is closely related with various types of factorisation of their symbols (\cite{GoKru92,MiPro86}).
The purpose of those factorisations may be to simplify the study of the properties of the corresponding Toeplitz operator by reducing it to that of a simpler operator, or to obtain a description of the Toeplitz kernel in terms of a multiplier acting on a model space, as in \eqref{eq:Hayashi_representation}, or to establish an equivalence between the existence of a certain type of factorisation of the symbol and certain properties of the Toeplitz operator.

The first question that arises when studying Toeplitz kernels is whether they are different from \zero.
A necessary and sufficient condition for $T_g$ not to be injective can be given in terms of a certain factorisation of its symbol, as follows.

\begin{prop}
[\cite{CMP21}]
\label{prop:ker.not.zero.factorisation}
We have $\ker T_g\neq\zero$ if and only if
\begin{equation}
\label{eq:ker.not.zero.factorisation}
g= \ol\OO \ol\alpha O^{-1}
\end{equation}
with $\alpha$ inner, non constant, and $O,\OO\in\Hp$, outer.
\end{prop}

\begin{proof}
This is an immediate consequence of the RH description of $\ker T_g$, \eqref{eq:RH nxn}.
If $f_+ = I_1 O_1$ and $f_- = \ol z \ol{I_2} \ol{O_2}$, where $I_1,I_2$ are inner and $O_1,O_2\in\Hp$ are outer, then
\[
g(I_1 O_1) = \ol z \ol{I_2}\ol{O_2} \iff g = \ol z \ol{I_1 I_2} \frac{\ol{O_2}}{O_1}.
\]
So $g$ admits a representation \eqref{eq:ker.not.zero.factorisation}.
Conversely, if $g= \ol\OO \ol\alpha O^{-1}$ then, by Proposition \ref{prop:O.k,O.ktil-in-the-kernel} below, $O k_\lambda^\alpha \in \ker T_g$ for any $\lambda \in\DD$, so $\ker T_g \neq\zero$.
\end{proof}

\begin{cor}[{\cite[Lemma 3.2]{MaPol05}}]
\label{cor:ker.not.zero.factorisation.unimodular}
If $|g|=1$ then $\ker T_g\neq\zero$ if and only if $g=\ol \alpha \ol O / O$ where $\alpha$ is inner, non constant, and $O\in\Hp$ is outer.
\end{cor}

The following result was used in the proof of Proposition \ref{prop:ker.not.zero.factorisation}.

\begin{prop}
\label{prop:O.k,O.ktil-in-the-kernel}
If \eqref{eq:ker.not.zero.factorisation} holds then, $O k_\lambda^\alpha, O \tilde k_\lambda^\alpha \in \ker T_g$, for any $\lambda\in\DD$.
\end{prop}

\begin{proof}
We have that $k_\lambda^\alpha\in\GG\Hinf$, so $O k_\lambda^\alpha \in \Hp$ and $g(O k_\lambda^\alpha) = \ol \OO \ol z \ol{\tilde k_\lambda^\alpha} \in \Hm$,
therefore $O k_\lambda^\alpha \in \ker T_g$.
On the other hand, $\tilde k_\lambda^\alpha \in \Hinf$, so $O \tilde k_\lambda^\alpha \in \Hp$ and 
\begin{equation}
\label{eq:proof.O.ktil-in-the-kernel}
g (O \tilde k_\lambda^\alpha) = \ol\OO \ol\alpha \ol{\tilde k_\lambda^\alpha} = \ol z \ol\OO \ol{k_\lambda^\alpha} \in \Hm,
\end{equation}
thus $O \tilde k_\lambda^\alpha \in \ker T_g$.
\end{proof}

Proposition \ref{prop:O.k,O.ktil-in-the-kernel} shows that, from a factorisation \eqref{eq:ker.not.zero.factorisation} for $g$, not only do we know that $\ker T_g\neq\zero$, but moreover we can identify families of functions belonging to $\ker T_g$ and establish lower bounds for the dimension of the Toeplitz kernel.
We have the following, where $FBP$ denotes the set of all finite Blaschke products.

\begin{prop}
\label{prop:lower.bounds.from.factorisation}
Let $g$ satisfy \eqref{eq:ker.not.zero.factorisation}.
Then
\begin{enumerate}[(i)]
\item if $\alpha\in FBP$, with exactly $k$ zeroes in \DD\ (counting multiplicities), then $\dim\ker T_g \geq k$, and the equality holds if and only if $O$ is square-rigid;

\item if $\alpha\notin FBP$ or $\dim\ker T_{\frac{\ol O}{O}} = \infty$, then $\ker T_g$ is infinite dimensional.
\end{enumerate}
\end{prop}

\begin{proof}
(i) was proved in \cite[Theorem 2.8]{CP18_multipliers}.

(ii) Since $O k_\lambda^\alpha \in \ker T_g$ for all $\lambda\in\DD$ by Proposition \ref{prop:O.k,O.ktil-in-the-kernel}, and, if $\alpha\notin FBP$, any set $\{ k_{\lambda_j}^\alpha \}$ is linearly independent, it folows that $\dim \ker T_g = \infty$.
The last part follows from 
\[
\ker T_\frac{\ol O}{O} \subset \ker T_{\ol\alpha\frac{\ol O}{O}} = \ker T_{\ol\OO\ol\alpha O^{-1}} = \ker T_g.
\]
\end{proof}

We can also formulate a necessary and sufficient condition for a representation of the form \eqref{eq:ker.not.zero.factorisation} to exist, in terms of maximal functions in $\ker T_g$, as follows.

\begin{prop}
\label{prop:not.zero.fact_max.func}
$g$ admits a factorisation \eqref{eq:ker.not.zero.factorisation} if and only if $\ker T_g$ has a maximal function of the form $f^M = O f_\alpha^M$ where $O\in\Hp$ is outer and $f_\alpha^M$ is a maximal function in the model space $K_\alpha$.
In that case, $\ol\alpha \ol O /O$ is a unimodular symbol for $\ker T_g$.
\end{prop}

\begin{proof}
If $g = \ol\OO \ol\alpha O^{-1}$ then $O \tilde k_\lambda^\alpha$ is a maximal function in $\ker T_g$, and $\tilde k_\lambda^\alpha$ is a maximal function in $K_\alpha$.
Conversely, if $O f_\alpha^M$ is a maximal function in $\ker T_g$, where $f_\alpha^M$ is a maximal function in $K_\alpha$, then $g O f_\alpha^M = \ol z \ol{\OO_1}$, with $\OO_1\in\Hp$ outer;
and, since by Theorem \ref{thm: maximal function characterization} we have that $\ol\alpha f_\alpha^M = \ol z \ol{\OO_2}$ with $\OO_2 \in\Hp$ outer, then
\[
g \alpha \ol{\OO_2} = O^{-1} \ol{\OO_1} \iff g = \frac{\ol{\OO_1}}{\ol{\OO_2}} \ol\alpha O^{-1}
\]
where $\OO_1 / \OO_2 \in \NN\cap L^2 = \Hp$ is outer.
\end{proof}

If $\ker T_g = O K_\alpha$, with $O$ and $\alpha$ as in Proposition \ref{prop:not.zero.fact_max.func}, then, for any maximal function $f^M_\alpha$ in $K_\alpha$, $O f^M_\alpha$ is a maximal function in $\ker T_g$, but the converse is not true.
Indeed, in Proposition \ref{prop:not.zero.fact_max.func}, the outer function $O$ is not necessarily square-rigid, so it is not, in general, a multiplier from a model space onto $\ker T_g$.
One may then ask when is the factor $O$ in \eqref{eq:ker.not.zero.factorisation} a multiplier from a model space onto $\ker T_g$.
We have the following, which is an immediate consequence of Theorem \ref{thm:multiplicadores_onto_kernels} (iii).

\begin{prop}
\label{prop:mult.from.max.funct}
Let $f_\alpha^M$ be a maximal function in $K_\alpha$, with $\alpha$ inner.
If $Of_\alpha^M$ is a maximal function in $\ker T_g$, with $O\in\Hp$ outer, then
\begin{equation}
\label{eq:mult.from.max.funct}
\ker T_g = O K_\alpha \iff OK_\alpha \subset L^2 \text{ and } O^{-1} \ker T_g \subset L^2.
\end{equation}
\end{prop}

\begin{ex}
Consider $\ker T_{\ol\alpha (z+1)^{1/5}}$ where $\alpha$ is an inner function which is analytic in a neighbourhood of $-1$.
We have that $f^M = \tilde k^\alpha_\lambda (z+1)^{-1/5}$ is a maximal function in that kernel, since $\ol\alpha (z+1)^{1/5} f^M = \ol\alpha \tilde k^\alpha_\lambda = \ol z \ol k^\alpha_\lambda$.
Since $\tilde k^\alpha_\lambda$ is maximal in $K_\alpha$ and $O := (z+1)^{-1/5}$ satisfies the conditions on the right-hand side of \eqref{eq:mult.from.max.funct}, we have that $\ker T_{\ol\alpha (z+1)^{1/5}} = (z+1)^{-1/5} K_\alpha$.
\end{ex}

Some particular types of factorisation of the form \eqref{eq:ker.not.zero.factorisation} for the symbol of a Toeplitz operator are of special interest, allowing us to characterise its kernel in terms of a multiplier acting on a model space, as in \eqref{eq:Hayashi_representation}, namely $L^2$-factorisation (and its generalisations) and maximal function factorisations.

\subsection{$L^2$-factorisation and Toeplitz operators with piecewise continuous symbols}
\label{subsec:L2_fact_PC_symbols}

We say that $g\in\Linf$ admits a {\em $L^2$-factorisation} (\cite[Chapter 2]{LitSpit87}) if and only if it can be represented as a product
\begin{equation}
\label{eq:L2_factorisation}
g=g_- z^n g_+ \quad , \text{ with } n\in\ZZ,\; g_+^{\pm 1} \in \Hp, \; g_-^{\pm 1} \in \ol\Hp.
\end{equation}
When $n=0$, the $L^2$-factorisation is said to be {\em canonical}.
It is easy to see that the following holds.

\begin{prop}
\label{prop:mult.from.L2-fact}
The factorisation \eqref{eq:L2_factorisation} is unique up to constant factors in $g_\pm$.
If such a factorisation exists, then $\ker T_g\neq\zero$ if and only if $n<0$;
in that case $g_+^{-1} z^{|n|-1}$ is a maximal function in $\ker T_g$ and we have
\begin{equation*}
\ker T_g = g_+^{-1} K_{z^{|n|}}.
\end{equation*}
\end{prop}

A $L^2$-factorisation \eqref{eq:L2_factorisation} such that
\begin{equation}
\label{eq: WH extra condition}
g_-P^+g_-^{-1}I : \mathcal{D}\to L^2 \text{ is bounded}
\end{equation}
for some $\mathcal D$ dense in $L^2$ (for instance, $\mathcal D$ may be the space \RR\ of all rational functions without poles on \TT) is called a {\em Wiener-Hopf factorisation} in $L^2$, also known as a {\em generalised factorisation} in $L^2$ (\cite[Chapter 4, \S 3]{MiPro86}).
This is a particularly important factorisation of the symbol, as we have the following equivalences (\cite{MiPro86}, see also \cite{Ca17}).

\begin{thm}
\label{thm:WH_fact_and_Fredholmness}
The operator $T_g$ is Fredholm if and only if $g$ admits a Wiener-Hopf factorisation in $L^2$;
$T_g$ is invertible if and only if that factorisation is canonical.
\end{thm}

We say that $g$ has a {\em bounded factorisation} if and only if, for some inner function $\alpha$,
\begin{equation}
\label{eq:bounded_factorisation}
g = g_- \alpha^n g_+ \quad , \text{ with } n\in\ZZ,\; g_+\in\GG\Hinf,\; g_-\in\GG\ol\Hinf.
\end{equation}
In that case we say that $g\sim\alpha^n$ (\cite[Definition 2.23]{CP18_multipliers}).

It is clear that, if \eqref{eq:bounded_factorisation} holds, then $T_g = T_{g_-} T_{\alpha^n} T_{g_+}$, where $T_{g_\pm}$ are invertible Toeplitz operators.
When $\alpha = z$, \eqref{eq:bounded_factorisation} is a {\em bounded Wiener-Hopf factorisation}, where \eqref{eq: WH extra condition} is necessarily satisfied with $\mathcal{D} = \Hp$.

It is known that a Wiener-Hopf factorisation in $L^2$ exists for a wide class of functions, including all non-vanishing continuous functions on \TT\ (\cite{GoKru92,MiPro86}), and that it is bounded when the functions belong to certain algebras of continuous functions, called decomposing algebras (\cite{MiPro86}).
Examples of such algebras are the Wiener algebra and the algebra $C^\mu$ of all H\"older continuous functions with exponent $\mu \in  ]0,1[$, which are particularly relevant in applications in physics and engineering (\cite{CamaraCardoso24,CamaraCardoso25,KiMi23,KiAbMiRo21}).
\RR\ is a subspace of any decomposing algebra, so any function in \GG\RR\ admits a bounded Wiener-Hopf factorisation.
In particular, if $B$ is a finite Blaschke product we can write
\begin{equation}
\label{eq:B-WH.fact}
B = B_- z^n B_+,
\end{equation}
with $B_- = \ol{B_+^{-1}}$, $B_+ \in \RR\cap\GG\Hinf$, where $n$ is the number of zeroes of $B$ in \DD.

However a non-vanishing piecewise continuous $g$ on \TT, i.e., with $\inf \{|g(t)| : t\in\TT \} > 0$, may not admit a $L^2$-factorisation, even in very simple cases, such as $g=z^{1/2}$ (\cite[Thm 2.5]{LitSpit87}).
Those functions, with discontinuities $c_k\in\TT$, $k=1,\dots,n$, and $g(z^\pm)\neq 0$ for all $z\in\TT$, can be represented in the form
\begin{equation}
\label{eq:piecewise.cont.decomposition}
g(z) = h(z) \prod_{k=1}^n (z^{\alpha_k})_{c_k}
\end{equation}
where $h\in \GG C(\TT)$, 
\begin{equation*}
\alpha_k = \frac{1}{2 \pi i} \log \frac{g(c_k^-)}{g(c_k^+)} \quad , \quad -\frac{1}{2} \leq \Re \alpha_k < \frac{1}{2},
\end{equation*}
and $(z^{\alpha_k})_{c_k}$ denotes a branch of $z^{\alpha_k}$ with branch cut crossing \TT\ at $c_k$.
The function g admits a $L^2$-factorisation if and only if $\Re \alpha_k \neq -\frac{1}{2}$ for all $k=1,\dots,n$ (\cite[Lemma 4.1]{Dudu79}).
We say in that case that $g$ is 2-regular.
So, if $\Re \alpha_k = -\frac{1}{2}$ for some $k$ (which geometrically means that the line segment with endpoints $g(c_k^-)$ and $g(c_k^+)$ contains the origin in the complex plane), then $g$ does not admit a $L^2$- factorisation.

One can however extend the result of Proposition \ref{prop:mult.from.L2-fact} to that class of symbols by using the following generalisation of the notion of $L^2$-factorisation.

\begin{thm}
\label{thm:fact.aux.polynomial}
Let $\gnotzero $, $g = g_- z^n g_+^{-1}$ with $n\in\ZZ$, $g_+,\ol{g_-} \in \Hp$ outer, such that, for some polynomial $P$ with $p$ simple zeroes ($p\geq 0$) all belonging to \TT, we have
\begin{equation}
\label{eq:fact.aux.polynomial}
Pg_+^{-1} \in \Hp \text{ and } \frac{g_+}{z-\lambda} \notin \Hp \text{ for every } \lambda \in \TT.
\end{equation}
Then, if $n\geq 0$, we have $\ker T_g = \zero$ and, if $n<0$, we have
\begin{equation}
\label{eq:mult.from.fact.aux.polynomial}
ker T_g = g_+ K_{z^{|n|}}.
\end{equation}
\end{thm}

\begin{proof}
We have that $\ker T_g = \ker T_{\ol{g_+} z^n g_+^{-1}}$ by Corollary \ref{cor:kernels_iguais}, so $\ker T_g$ is given by $f_+ \in \Hp$ such that $\ol{g_+} z^n g_+^{-1} f_+ = f_-$ with $f_-\in \Hm$.
We have
\begin{align}
\ol{g_+} z^n g_+^{-1} f_+ = f_- & \iff z^n P g_+^{-1} f_+ = \frac{P}{\ol P} \, \ol{P g_+^{-1}} f_- \nonumber \\
		& \iff z^n P g_+^{-1} f_+ = c z^p \, \ol{P g_+^{-1}} f_-
		\label{eq:proof_fact.aux.polynomial}
\end{align}
with $c\in\CC$, where we used the equality $\frac{P}{\ol P} = cz^p$.
If $n\geq 0$ we have
\[
\annot{z^n P g_+^{-1} f_+}{\in (\Hp)^2}
=
c z^p
\annot{P g_+^{-1}}{\in\ol\Hp}
\annot{f_-}{\in \Hm}
= Q_{p-1}
\]
where $Q_{p-1}$ is a polynomial of degree at most $p-1$ (cf. \cite{CMP16} in the context of $L^2(\mathbb{R})$).
Therefore $z^n f_+ = \frac{g_+}{P}\, Q_{p-1}$
but, since $\frac{g_+}{P}$ is not square-integrable in any neighbourhood of any zero of $P$, the polynomial $Q_{p-1}$ must vanish at all the $p$ zeroes of $P$ and it follows that $Q_{p-1}=0$.

If $n<0$, we have from \eqref{eq:proof_fact.aux.polynomial}, following a similar reasoning,
\[
Pg_+^{-1} f_+ = c z^{|n|+p} \ol{Pg_+^{-1}} f_- = Q_{|n|+p-1}
\]
where $Q_{|n|+p-1}$ is a polynomial of degree at most $|n|+p-1$.
Since $Q_{|n|+p-1}$ must vanish at all zeroes of $P$, we get that
$g_+^{-1} f_+ = \tilde Q_{|n|-1} \in K_{z^{|n|}}$.
\end{proof}

\begin{rem}
Similarly to $L^2$-factorisation, the factorisation in Theorem \ref{thm:fact.aux.polynomial} is unique up to constant factors in $g_\pm$.
Note also that a $L_2$-factorisation \eqref{eq:L2_factorisation} satisfies conditions \eqref{eq:fact.aux.polynomial} with $P=1$.
\end{rem}

\begin{ex}
\label{ex:g=+-1}
Consider $g$ such that $g(z) = 1$ if $z \in \TT\cap\CCp$, and $g(z)=-1$ if $z \in \TT\cap\CCm$, where \CCp\ (respectively, \CCm) denotes the upper (respectively, lower) halp-plane.
We can write, with the notation of \eqref{eq:piecewise.cont.decomposition},
$$g = z (z^{-\frac{1}{2}})_{-1} (z^{-\frac{1}{2}})_{1}= g_- z g_+^{-1}$$ 
with $g_- = (1 + \ol z)^{1/2} (1 - \ol z)^{1/2}$ and $g_+ = (z+1)^{1/2} (z-1)^{1/2}$.
Taking $P=(z+1)(z-1)$ and applying Theorem \ref{thm:fact.aux.polynomial}, we get that $\ker T_g = \zero$, since $n=1 \geq 0$.
\end{ex}

Note that, using Corollary \ref{cor:symbol_zero_on_T}, one can extend the previous results to some classes of symbols that vanish on \TT.

We now apply Theorem \ref{thm:fact.aux.polynomial} to obtain a model space representation for Toeplitz kernels with piecewise continuous symbols that do not admit a $L^2$-factorisation.
Furthermore, we show that one can obtain from \eqref{eq:mult.from.fact.aux.polynomial} an isometric model space representation of $\ker T_g$.

\begin{ex}
\label{ex:Finite_dimensional_Toeplitz_kernels}
A class of finite dimensional Toeplitz kernels
\end{ex}

\vspace{-8pt}

We now consider the question of characterising the kernels of a class of Toeplitz operators with symbols of the form $\ol z^{n/2}$ where $n\in\NNN$ is odd.
These symbols do not admit a $L^2$-factorisation (so the associated Toeplitz operator is not Fredholm), but they admit a factorisation as in Theorem \ref{thm:fact.aux.polynomial}.
This yields a representation of $\ker T_g$ of the form \eqref{eq:mult.from.fact.aux.polynomial}.
We show here how this allows for determining an isometric multiplier from a model space onto the Toeplitz kernel and obtain Hayashi's representation \eqref{eq:Hayashi_representation} in a simple and explicit way.

Let $\ol z^{n/2} = \ol z^{(n-1)/2} \; \ol z^{1/2}$ where we take the principal branch of the square-root, so that $\ol z^{1/2}$ is discontinuous on \TT\ at the point $-1$.
We have then the factorisation
\begin{equation}
\label{eq:z^n/2_fact_aux_polyn}
\ol z^\frac{n}{2} = \left( \frac{z+1}{z} \right)^\frac{1}{2} z^{-\frac{n-1}{2}} \frac{1}{(z+1)^\frac{1}{2}}
\end{equation}
where the factors $g_- = \left( \frac{z+1}{z} \right)^{1/2}$ and $g_+ = (z+1)^{1/2}$ satisfy the conditions \eqref{eq:fact.aux.polynomial} with $P(z) = z+1$.

By Theorem \ref{thm:fact.aux.polynomial}, $\ker T_{\ol z^{n/2}} = \zero$, for $n=1$, and
\begin{equation}
\label{eq:mult_Kz^n/2}
\ker T_{\ol z^\frac{n}{2}} = (z+1)^\frac{1}{2} K_{z^N}, \quad N= \frac{n-1}{2} \text{, for } n \geq 3 \text{, odd}.
\end{equation}

To obtain an isometric multiplier from a model space onto $\ker T_{\ol z^{n/2}}$, we can take advantage of the fact that \eqref{eq:mult_Kz^n/2} provides a very simple description of that Toeplitz kernel and yields a basis
\begin{equation}
\BB = \{ (z+1)^\frac{1}{2}, (z+1)^\frac{1}{2} z, \dots , (z+1)^\frac{1}{2} z^{N-1} \},
\end{equation}
from which one can easily determine a basis for the subspace of all functions in the Toeplitz kernel which vanish at $0$:
\begin{equation}
\BBo = \{ (z+1)^\frac{1}{2} z, \dots , (z+1)^\frac{1}{2} z^{N-1} \}.
\end{equation}

The isometric multiplier in Hayashi's representation of $\ker T_{\ol z^{n/2}}$ will be a constant multiple of the orthogonal vector to \BBo\ in that kernel (\cite{Sa94}) which can be obtained by subtracting, from the element $(z+1)^{1/2} \in\BB$, its orthogonal projection into the space with basis \BBo.

To illustrate this result we take a particular case, with $n=7$.
We have
\begin{equation}
\label{eq:mult_Kz^7/2}
\ker T_{\ol z^\frac{7}{2}} = (z+1)^\frac{1}{2} K_{z^3}.
\end{equation}
With the notation above, $\BB = \{ (z+1)^{1/2} z^j, j=0,1,2 \}$ and $\BBo = \{ (z+1)^{1/2} z, (z+1)^{1/2} z^2 \}$.
The function
\begin{equation*}
\UU = (z+1)^\frac{1}{2} - P_\BBo (z+1)^\frac{1}{2} = (z+1)^\frac{1}{2} (1 - a z) (1 - \ol a z),
\end{equation*}
with $a = (1 + 2 i)/5$, is orthogonal to $\spn\BBo = z \ker T_{z \ol z^{7/2}}$, where $P_\BBo$ is the orthogonal projection onto $\spn\BBo$.
So,
\[
u = \frac{\UU}{||\UU||_2}
\]
is an isometric multiplier from a model space of the form $K_{z \alpha}$ onto $\ker T_{\ol z^{7/2}}$.

To determine the inner function $\alpha$, we use \eqref{eq:simbolo_vindo_de_multiplicador} in Theorem \ref{thm:multiplicadores_onto_kernels} together with the representation \eqref{eq:mult_Kz^7/2}:
\[
\ker T_{\ol z^\frac{7}{2}} = (z+1)^\frac{1}{2} K_{z^3} = \frac{\UU}{||\UU||_2} K_{z\alpha}
\]
implies that
\[
K_{z\alpha} = \frac{||\UU||_2}{\UU} (z+1)^\frac{1}{2} K_{z^3} = \frac{||\UU||_2}{(1 - a z) (1 - \ol a z)} K_{z^3} = K_{z B_a B_{\ol a}}
\]
with $B_\lambda (z) = \frac{z-\lambda}{1- \ol\lambda z}$.
Thus, Hayashi's representation for $\ker T_{\ol z^{7/2}}$ is given by
\[
\ker T_{\ol z^\frac{7}{2}} = \frac{(z+1)^\frac{1}{2} (1 - a z) (1 - \ol a z)}{||(z+1)^\frac{1}{2} (1 - a z) (1 - \ol a z)||_2} K_{z B_a B_{\ol a}}.
\]

\begin{rem}
Note that the same method can be applied to any finite dimensional Toeplitz kernel to obtain a representation of the form \eqref{eq:Hayashi_rep_v2}, once a representation of the form $\ker T_g = w K_{z^n}$ has been found.
\end{rem}

\subsection{Maximal function factorisation of symbols}
\label{subsec:Maximal_function_factorisation_of_symbols}

Any maximal function $f^M$ in $\ker T_g$, with inner-outer factorisation $f^M = IO$, where $I$ is inner and $O$ is outer, induces, by \eqref{eq:f=IO.max.gives.unimodular.symbol}, a factorisation of the symbol
\begin{equation}
\label{eq:fact.from.max.funct}
g = \ol{zI} \frac{\ol O}{O} \ol h \ , \quad \text{ with } h \in\Hinf \text{ outer},
\end{equation}
where we used Corollary \ref{cor:kernels_iguais}.
We call such a representation a {\em maximal function factorisation} of $g$.
For $|g|=1$, it takes the form
\begin{equation}
\label{eq:fact.from.max.funct.unimodular}
g = \ol{zI} \frac{\ol O}{O}.
\end{equation}
This factorisation is unique if we fix the outer factor $O$.

For example, if $g=\ol\theta$ with $\theta$ inner and $\ker T_g = K_\theta$, then \kktilo\ is a maximal function in $K_\theta$ with inner-outer factorisation $\kktilo = IO$ where
\[
I = \frac{\kktilo}{\kko} = \ol z \frac{\theta - \theta(0)}{1 - \ol{\theta(0)} \theta}
\]
is inner and $O = \kko = 1 - \ol{\theta(0)} \theta$ is outer ($O\in\GG\Hinf$).
It follows that any inner function $\theta$ can be factorised as
\begin{equation}
\label{eq:theta_max_fact}
\theta = \left( z \frac{\kktilo}{\kko} \right) \frac{\kko}{\ol\kko}.
\end{equation}

It may be useful to replace $z$ in \eqref{eq:fact.from.max.funct.unimodular} by its expression in terms of a Blaschke factor $B_\lambda$ with $\lambda\in\DD$, using the equality $z = B_\lambda \frac{1-\ol\lambda z}{1 - \lambda \ol z}$:
\begin{equation}
\label{eq:modified.max.func.fact}
g = \ol{B_\lambda I} \frac{\ol O (1- \lambda \ol z)}{O (1-\ol\lambda z)}.
\end{equation}
We say that \eqref{eq:modified.max.func.fact} is a {\em modified maximal function factorisation} of $g$.

\vspace{8pt}

If $\ker T_g = u K_{z\theta}$ is Hayashi's representation of a nontrivial Toeplitz kernel, then we see that the isometric multiplier $u$ is the outer factor of a certain maximal function in $\ker T_g$, $f^M = \theta u$.
We may then ask: if $f^M = IO$ is an inner-outer factorisation of a maximal function $f^M$ in $\ker T_g$, {\em when is the outer factor $O$ of the maximal function a multiplier from a model space onto $\ker T_g$} and, if so, whether there is a relation between that model space and the model space $K_I$ defined by the inner factor of $f^M$.
We have the following.

\begin{prop}
\label{prop:model.space.from.maximal.function}
Let $f^M=IO$ be an inner-outer factorisation of a maximal function in $\ker T_g$.
If $\ker T_g = O K_\alpha$ for some inner $\alpha$, then $K_\alpha = K_{zI}$.
\end{prop}

\begin{proof}
We have that $\ker T_g = \ker T_{\ol z \ol{I O} / O} = O K_\alpha = \ker T_{\ol\alpha \ol O /O}$.
So, by Corollary \ref{cor:kernels_iguais}, $\ol{zI} = \lambda \ol\alpha$ with $\lambda\in\CC$, $|\lambda|=1$.
\end{proof}

From Proposition \ref{prop:mult.from.max.funct} we also have:

\begin{thm}
\label{thm:multiplier.from.maximal.function}
Let $f^M=IO$ be an inner-outer factorisation of a maximal function in $\ker T_g$.
Then,
\[
\ker T_g = O K_{zI} \iff O K_{zI} \subset L^2 \text{ and } O^{-1} \ker T_g \subset L^2.
\]
\end{thm}

\begin{cor}
\label{cor:invertible.O_or_I.FBP_give.multiplier}
With the same assumptions as in Theorem \ref{thm:multiplier.from.maximal.function}, we have $\ker T_g = O K_{zI}$ if $O\in\GG\Hinf$, or if $O$ is square-rigid and $\dim K_I < \infty$.

\end{cor}

\begin{proof}
The result is clear if $O \in \GG\Hinf$, by Theorem \ref{thm:multiplier.from.maximal.function}.
Now, if $\dim K_I < \infty$, then $\dim K_{zI} < \infty$, so $K_{zI} \subset \Hinf$ and $O K_{zI} \subset \Hp$.
On the other hand, since $\ker T_{\ol z \ol O /O} = \spn\{ O \}$ if $O$ is square-rigid, it follows from \cite[Theorem 6.2]{CMP16} that $\dim \ker T_{\ol z \ol I \ol O/ O} = \dim K_I + 1$.
Since $O K_{zI} \subset \ker T_{\ol z \ol I \ol O  / O} = \ker T_g$
and
\[
\dim O K_{zI} = \dim K_{zI} = \dim K_I +1 = \ker T_{\ol z \ol I \frac{\ol O}{O}},
\]
we conclude that $\ker T_g = O K_{zI}$.
\end{proof}

\begin{ex}
Let $E$ be the singular inner function $E(z) = \exp\left( \frac{z+1}{z-1} \right)$ and let $g=\ol E z$.
By Theorem \ref{thm: maximal function characterization}, a maximal function in $\ker T_{\ol E z}$ is
\[
f^M = \frac{k^E_0 (0) \tilde k^E_0 - \tilde k^E_0 (0) k^E_0}{z},
\]
since
\[
\ol E z f^M = \ol z \tilde k^E_0 (0) \ol{k^E_0} \ol{ \left( \frac{k^E_0 (0)}{\tilde k^E_0 (0)} - \frac{\tilde k^E_0}{k^E_0} \right) },
\]
where $\alpha = \tilde k^E_0 / k^E_0$ is an inner function, $k^E_0 (0) = 1- 1/e^2$, $\tilde k^E_0 (0) = E'(0) = 1/e$ and
\[
k^E_0,\ \frac{k^E_0 (0)}{\tilde k^E_0 (0)} - \frac{\tilde k^E_0}{k^E_0} \in \GG\Hinf.
\]

An inner-outer factorisation of $f^M$ (cf. Section \ref{sec: Maximal functions and conjugations}) is $f^M = IO$ with
\begin{equation}
\label{eq:IO_fact_Ez}
I = \ol z \frac{k^E_0 (0) \tilde k^E_0 - \tilde k^E_0 (0) k^E_0}{\ol{k^E_0 (0)} k^E_0 - \ol{\tilde k^E_0 (0)} \tilde k^E_0}
\ , \ 
O = \ol{k^E_0 (0)} k^E_0 - \ol{\tilde k^E_0 (0)} \tilde k^E_0 \in \GG\Hinf,
\end{equation}
so, by Corollary \ref{cor:invertible.O_or_I.FBP_give.multiplier} we obtain a model space representation for $\ker T_{\ol E z}$ as
\[
\ker T_{\ol E z} = \ker T_{\ol z \ol I \frac{\ol O}{O}} = O K_{zI}.
\]

Noting that $O$ is the product of two functions in \GG\Hinf,
\[
O = 
\ol{k^E_0 (0)} k^E_0 \left( 1 - \ol{ \left( \frac{\tilde k^E_0 (0)}{k^E_0 (0)} \right) } \frac{\tilde k^E_0}{k^E_0} \right) = 
\ol{k^E_0 (0)} k^E_0 k^\alpha_0,
\]
and, from \eqref{eq:IO_fact_Ez}, $I = \tilde k^\alpha_0 / k^\alpha_0$, we can also write
\[
\ker T_{\ol E z} = \ker T_{
\annot{\ol z \ol{ \left( \frac{\tilde k^\alpha_0}{k^\alpha_0} \right) e} \frac{\ol{ k^\alpha_0}}{k^\alpha_0}
}{=\ \ol\alpha \text{ by \eqref{eq:theta_max_fact}}
}
\frac{\ol k^E_0}{k^E_0}} = k^E_0 K_\alpha
\]
Now, by \cite[Theorem 10]{Cro94}, $m = \sqrt{1 - |E(0)|^2} / k^E_0$ is an isometric multiplier from $K_E$, which contains $\ker T_{\ol E z}$, onto $K_{z \tilde k^E_0 / k^E_0} = K_{z\alpha}$, which contains $K_\alpha$, so $m$ is also an isometric multiplier from $\ker T_{\ol E z}$ onto $K_\alpha$, i.e.,
\[
\ker T_{\ol E z} = \frac{k^E_0}{\sqrt{1 - |E(0)|^2}} K_\alpha
\]
is an isometric model space representation of $\ker T_{\ol E z}$.
\end{ex}

A crucial role is played in the previous results by the inner-outer factorisation of maximal functions in a Toeplitz kernel.
This turns out to have a surprising connection, which we study in the next section, with conjugations that are naturally defined on Toeplitz kernels.

%%%%%%%%%%%%%       Section 4      %%%%%%%%%%%%%
%%%%%  Maximal functions and conjugations  %%%%%

\section{Maximal functions and conjugations}
\label{sec: Maximal functions and conjugations}

We assume throughout this section that $|g|=1$.
Recall that every nontrivial Toeplitz kernel can be described in terms of a unimodular symbol.

A conjugation on a Hilbert space \HH\ is an antilinear operator \CCC\ on \HH\ such that $\CCC^2 = I$, where $I$ denotes the identity operator, and $\langle \CCC f_1 , \CCC f_2 \rangle = \langle f_2 , f_1 \rangle$ for all $f_1, f_2 \in \HH$.
For $g$ with $|g|=1$, a conjugation $\CCC = \Cg$ can be defined by
\begin{equation}
\label{eq:def.conj.Cg}
\Cg f = \ol g \ol z \ol f, \text{ for all } f \in \ker T_g,
\end{equation}
which can be seen as the natural conjugation on $\ker T_g$ (\cite{DyPlaPtak22}).
It is clear from \eqref{eq:def.conj.Cg} that $|\Cg f| = |f|$ for all $f \in \ker T_g$.

For two unimodular symbols corresponding to the same Toeplitz kernel, the conjugations defined by \eqref{eq:def.conj.Cg} for those symbols differ by a unimodular constant.

Using the conjugation \Cg\ we can reformulate the characterisation of maximal functions presented in Theorem \ref{thm: maximal function characterization}, as being the conjugates, by \Cg, of the outer functions in $\ker T_g$.

\begin{thm}
\label{thm:max_func_charact_with_conj}
Let $f\in\Hp$; $f$ is a maximal function in $\ker T_g$ if and only if $\Cg f$ is outer in \Hp.
\end{thm}

\begin{proof}
Let $\OO$ denote an outer function in \Hp.
\[
\Cg f = \OO \iff \ol{gzf} = \OO \iff gf = \ol z \ol\OO
\]
which is equivalent by Theorem \ref{thm: maximal function characterization} to $f$ being a maximal function in $\ker T_g$.
\end{proof}

\begin{rem}
If $\theta$ is an inner function, we say that $f\in K_\theta$ generates $K_\theta$ if and only if 
\[
\bigvee \{ \left( S^* \right)^n f : n \geq 0 \} = K_\theta,
\]
where $S^*$ denotes the backward shift on \Hp.
By \cite[Corollary 8.25]{GarMashRoss16}, if $f\in K_\theta$ is outer then $C_\theta f$ generates $K_\theta$.
As a consequence of Theorem \ref{thm:max_func_charact_with_conj}, this is equivalent to saying that every maximal function in  $K_\theta$ generates $K_\theta$.
\end{rem}

Theorem \ref{thm:max_func_charact_with_conj} allows us to obtain, in a simple way, an inner-outer factorisation of any maximal function.

\begin{prop}
\label{prop:IO_fact_max_func_with_conj}
If $f^M$ is a maximal function in $\ker T_g$, then an inner-outer factorisation of $f^M$ is $f^M = IO$, with $I = \frac{f^M}{\Cg f^M}$ and $O = \Cg f^M$.
We have that $g f^M = \ol z \ol O$ where $O$ is the outer factor in an inner-outer factorisation of $f^M$.
\end{prop}

\begin{proof}
We have $f^M = \frac{f^M}{\Cg f^M} \, \cdot \, \Cg f^M$, where $\Cg f^M$ is outer by Theorem \ref{thm:max_func_charact_with_conj} and $\frac{f^M}{\Cg f^M}$ is a function in \NN, with modulus one, so it is an inner function.
\end{proof}

\begin{cor}
\label{cor:outer+conj=max_func}
If \OO\ is outer and $\OO \in \ker T_g$, then \Cg\OO\ is a maximal function in $\ker T_g$ with outer factor \OO\ (up to a unimodular constant).
If two maximal functions have the same outer factor then they differ by a unimodular constant.
\end{cor}

\begin{proof}
This follows from $\Cg^2 = I$, Theorem \ref{thm:max_func_charact_with_conj} and Proposition \ref{prop:IO_fact_max_func_with_conj}.
\end{proof}

In the case where the Toeplitz kernel is a model space $K_\theta$, one can associate to any outer function $F\in K_\theta$ an inner function, $I_F = \frac{C_\theta F}{F}$, which is called the associated inner function of $F$ with respect to $\theta$ (\cite[Definition 8.18]{GarMashRoss16}).
It follows from Proposition \ref{prop:IO_fact_max_func_with_conj} and Corollary \ref{cor:outer+conj=max_func} that $I_F$ is the inner factor of the maximal function with outer factor $F$ (up to a unimodular constant).

\begin{rem}
Note that two linearly independent maximal functions can share the same inner factor in an inner-outer factorisation.
For example, $f_1 (z) = z-1$ and $f_2(z) = z+1$ are both maximal functions  in $K_{z^2}$, with inner factor $I=1$.
\end{rem}

The set of all elements in $\ker T_g$ with the same outer factor as a given maximal function is characterised as follows, generalising an analogous result for model spaces (\cite[Proposition 8.20]{GarMashRoss16}).

\begin{prop}
Let $f^M$ be a maximal function in $\ker T_g$ with inner-outer factorisation $f^M = IO$.
The set of all functions in $\ker T_g$ with outer factor $O$ is
\[
\{ \alpha O : \alpha \text{ inner and } \alpha \preceq I \}.
\]
\end{prop}

\begin{proof}
If $\alpha O \in \ker T_g$, then by \eqref{eq:f=IO.max.gives.unimodular.symbol} $(\ol z \ol I \ol O / O) (\alpha O) = \ol z \ol{\psi_+}$ with $\psi_+ \in \Hp$, so $\ol I \alpha = \ol{\psi_+} / \ol O \in \ol\NN \cap \Linf = \ol\Hinf$, implying that $\alpha \preceq I$.
Conversely, if $\alpha \preceq I$, then $(\ol z \ol I \ol O / O) (\alpha O) = \ol z \ol I \alpha \ol O \in \Hm$, so $\alpha O \in \ker T_g$. 
\end{proof}

In every nontrivial Toeplitz kernel, there are maximal functions whose outer factor is square-rigid \cite{CMP16}.
It is thus natural to ask whether there exists a maximal function whose outer factor is $1$ (that is, an inner maximal function), or whose inner factor is a unimodular constant (that is, an outer maximal function), or even whether one can construct, from a maximal function, other maximal functions with different outer factors.
Note that, by Corollary \ref{cor:outer+conj=max_func}, a maximal function $f^M$ is defined, up to a unimodular constant, by its outer factor $\Cg f^M$.

The answer to the first question is as follows.

\begin{prop}
\label{prop:inner_maximal_function}
There exists a function $f\in\ker T_g$ such that $\Cg f = 1$, i.e., $\ker T_g$ has an inner maximal function $f^M = \theta$, if and only if $\ker T_g = K_{z \theta}$.
\end{prop}

\begin{proof}
$\Cg \theta = 1 \iff \ol{gz\theta} = 1 \iff g=\ol{z\theta}$.
\end{proof}

In general a maximal function $f^M$ is neither inner nor outer, but one can construct from $f^M$ other maximal functions related by a partial order.
Let $f_1^M$ and $f^M_2$ be maximal functions in $\ker T_g$ and let $f_j^M = I_j O_j$, $j=1,2$, be their inner-outer factorisations ($I_j$ inner, $O_j$ outer).

\begin{df}
\label{def:partial_order_max_functs}
We say that 
\begin{equation}
\label{eq:def:partial_order_max_functs}
f_1^M \preceq f^M_2 \text{ if and only if } I_1 \preceq I_2
\end{equation}
(where $I_1 \preceq I_2$ means that $I_2 \ol{I_1} \in \Hinf$).
If $I_1 \prec I_2$, i.e., $I_2 \ol{I_1} \in \Hinf$ is not a constant, then we say that $f_1^M \prec f^M_2$.
\end{df}

The relation defined by \eqref{eq:def:partial_order_max_functs} can also be expressed in terms of the outer factors $O_j$, $j=1,2$.

\begin{prop}
\label{prop:partial_order_outer_kernels}
$f_1^M \preceq f_2^M \iff \ker T_{\ol z \ol{O_2} / O_2} \subset \ker T_{\ol z \ol{O_1} / O_1}$.
\end{prop}

\begin{proof}
Since both $f_1^M$ and $f_2^M$ are maximal functions in $\ker T_g$, by \eqref{eq:fact.from.max.funct.unimodular} we have that
\[
\ol{I_1} \frac{\ol{O_1}}{O_1} = c \ol{I_2} \frac{\ol{O_2}}{O_2}, \text{ with } c\in\CC, |c|=1,
\]
and $\ol z \frac{\ol{O_1}}{O_1} = (c \ol{I_2} I_1)\, \ol z \, \frac{\ol{O_2}}{O_2}$.
If $f_1^M \preceq f_2^M$, then $c \ol{I_2} I_1 \in \ol\Hinf$, so $\ker T_{\ol z \frac{\ol{O_2}}{O_2}} \subset \ker T_{\ol z \frac{\ol{O_1}}{O_1}}$.
Conversely, if the latter inclusion holds, then by Corollary \ref{cor:kernel.subset.of.another.kernel} we have that
\[
\frac{O_2}{\ol{O_2}} \frac{\ol{O_1}}{O_1} \in \ol\NN \implies c \ol{I_2} I_1 \in \ol\NN\cap\Linf = \ol\Hinf,
\]
so $I_1 \preceq I_2$.
\end{proof}

Noting that $O_j = \Cg f_j^M$, $j=1,2$ (up to a unimodular constant) and $\ker T_{\ol z \ol{O_j} / O_j} = \KKmin (O_j)$, we see that the relation \eqref{eq:def:partial_order_max_functs} between maximal functions can moreover be expressed in terms of the conjugation \Cg:
\begin{equation*}
\label{eq:partial_order_minimal_kernel}
f_1^M \preceq f_2^M \iff \KKmin(\Cg f_1^M) \supset \KKmin(\Cg f_2^M).
\end{equation*}
It is easy to see that, if $f_1^M$ is outer ($I_1\in\CC$) then there exists no maximal function $f^M$ in $\ker T_g$ such that $f^M \prec f_1^M$ and, if $f_2^M = I_2 O_2$ with $O_2$ square-rigid, then there exists no maximal function $f^M$ such that $f_2^M \prec f^M$.

\vspace{8pt}

Now we consider the question of obtaining, from a given maximal function in $\ker T_g$, other maximal functions related to it by the partial order \eqref{eq:def:partial_order_max_functs}.
In what follows we assume that $f^M$ is a maximal function in $\ker T_g$, with inner-outer factorisation $f^M = I O$.

\begin{prop}
\label{prop:maximal_function_prescribed_inner_factor}
For every inner function $\alpha$ such that $\alpha \prec I$, there exists a maximal function $f_\alpha^M$ in $\ker T_g$, with inner-outer factorisation $f_\alpha^M = \alpha\OO$, where $\OO\in\Hp$ is outer, given by
\begin{equation}
\label{eq:outer_factor_max_func_prescribed_inner_factor}
\OO = O (\lambda + I \ol \alpha ), \text{ with } \lambda \in\CC, |\lambda| = 1.
\end{equation}

We have then $f_\alpha^M \prec f^M$ and $\ker T_{\ol z \frac{\ol\OO}{\OO}} = \ker T_{\ol z \alpha \ol I \frac{\ol O}{O}} \supsetneq \ker T_{\ol z \frac{\ol O}{O}}$.
\end{prop}

\begin{proof}
For any $\lambda \in \CC$, $|\lambda|=1$, $\lambda + I \ol\alpha$ is outer in \Hinf, so $\OO$ is outer and
\begin{align*}
\Cg ( \alpha\OO) &= \Cg \left(\alpha O \left(\lambda + I \ol\alpha\right) \right) = \ol {gz} \ol O \left(\ol{ I + \lambda \alpha}\right) \\
&= \left(z I \frac{O}{\ol O}\right) \ol z \ol O \left(\ol{ I + \lambda \alpha}\right) = \ol\lambda O \left(\lambda + I \ol\alpha\right)
\end{align*}
where the right-hand side represents an outer function in \Hp.
Thus, by Theorem \ref{thm:max_func_charact_with_conj}, $f_\alpha^M = \alpha O$ is a maximal function in $\ker T_g$ such that $f_\alpha^M \prec f^M$ and the last inclusion follows from the equality
\[
\frac{\ol\OO}{\OO} = \frac{\ol O (\ol\lambda - \ol I \alpha)}{O (\lambda - I \ol\alpha)} = \frac{\ol O}{O} \ol\lambda \ol I \alpha.
\]
\end{proof}

Taking $\alpha=1$ in Proposition \ref{prop:maximal_function_prescribed_inner_factor}, we see that
\begin{equation}
\OO^M = O (\lambda + I )
\end{equation}
is an outer maximal function in $\ker T_g$, for any $\lambda\in\CC$, $|\lambda|=1$.
Thus we get the following

\begin{cor}
\label{cor:outer_max_func}
Every nontrivial Toeplitz kernel has a maximal function $O^M$ which is outer and we can write
\begin{equation}
\label{eq:outer_max_func}
\ker T_g = \ker T_{\ol z \frac{\ol{O^M}}{O^M}}.
\end{equation}
\end{cor}

\begin{rem}
\label{rem:kernel_symbol_outer_Hayashi}
Hayashi showed in \cite[Lemma 5]{Ha85} that, for every Toeplitz operator $T_g$ with non-zero kernel, there exists an outer function $h\in\Hp$, such that $\ker T_g = \ker T_{\ol h /h}$.
If we write $z$ as a quotient of an outer function by its complex conjugate, $z=(z - \eta) / (\ol z - \ol\eta)$, with $\eta\in\TT$, we see that Corollary \ref{cor:outer_max_func} also provides a representation for $\ker T_g$ in the form $\ker T_{\ol h /h}$ where $h$, which in general does not belong to the Toeplitz kernel, can be obtained from any maximal function $f^M$ in a simple way.
From \eqref{eq:outer_max_func} we can write, for instance, $\ker T_g = \ker T_{\ol h /h}$ with $h= (z+1) (\Cg f^M + f^M)$.
\end{rem}

As a consequence of the previous results we have the following.

\begin{cor}
For any maximal function $f^M$ in $\ker T_g$ there exist outer functions $O_1,O_2\in\ker T_g$, with $O_2$ square-rigid, such that
\begin{equation*}
O_1 = \Cg O_1 \preceq f^M \preceq \Cg O_2.
\end{equation*}
\end{cor}

From Theorem \ref{thm:max_func_charact_with_conj} and Corollary \ref{cor:outer_max_func} we see moreover that there are elements of $\ker T_g$ on which the conjugation \Cg\ acts as the identity or a constant multiple of the identity, such as the outer maximal functions in $\ker T_g$.
For those functions we have $\Cg f = \lambda f$, $f\neq 0$, where we must of course have $|\lambda|=1$.
If $\Cg f = \lambda f$, writing $\lambda = a / \ol a$ with $a\in\CC$ we have that $\Cg (af) = af$, so $af$ is invariant under \Cg.
One may then ask which other elements of $\ker T_g$ are eigenfunctions for some eigenvalue $\lambda$.
The next proposition describes all such functions in terms of maximal functions.

\begin{thm}
\label{thm:Cg_eigenfunctions}
Let $f\in\ker T_g$ and let $f=\alpha O$ be an inner-outer factorisation.
Then the following are equivalent:
\begin{enumerate}[(i)]
\item $\Cg f = \lambda f$ for some $\lambda\in\CC$, $|\lambda|=1$;

\item $\alpha f$ is a maximal function in $\ker T_g$,

\item $\alpha \KKmin (f) \subset \ker T_g$ and $\ker T_g$ is the minimal Toeplitz kernel containing $\alpha \KKmin (f)$.
\end{enumerate}
\end{thm}

\begin{proof}
$(i) \iff (ii)$
We have that 
\begin{equation*}
\Cg f = \lambda f \iff \Cg (\alpha O) = \lambda \alpha O \iff \Cg (\alpha^2 O) = \lambda O
\end{equation*}
which, by Theorem \ref{thm:max_func_charact_with_conj}, is equivalent to $\alpha^2 O$ being a maximal function in $\ker T_g$.

$(ii)\implies (iii)$
Since $\alpha \KKmin (f) \subset \Hp$ and $\alpha f \in \ker T_g$, by Theorem \ref{thm:multiplicadores_into_kernels} we have that $\alpha \KKmin (f) \subset \ker T_g$.
Moreover, if $\alpha \KKmin (f) \subset \ker T_h$, for some $h\in\Linf\bs\zero$, then $\ker T_h \supset \KKmin (\alpha f) = \ker T_g$, so $\ker T_g$ is the minimal kernel containing $\alpha \KKmin (f)$.

$(iii)\implies (ii)$ Let $(iii)$ hold, then $\alpha f \in \ker T_g$.
If $\alpha f \in \ker T_h$ for some $h\in\Linf\bs\zero$, then $\alpha \KKmin (f) \subset \ker T_h$ by Theorem \ref{thm:multiplicadores_into_kernels}, since $f$ is a maximal function in $\KKmin (f)$ and $\alpha \KKmin (f) \subset L^2$.
Therefore, $\ker T_g \subset \ker T_h$, as $\ker T_g$ is the minimal Toeplitz kernel containing $\alpha \KKmin (f)$.
This means that $\KKmin (\alpha f) = \ker T_g$.
\end{proof}

\begin{ex}
Two simple examples in $K_{z^5}$:
we see that $(1-z)^4$ is an outer maximal function and an eigenvector of $C_{\ol z^5}$ (for the eigenvalue 1) since $C_{\ol z^5} (1-z)^4 = (1-z)^4$;
we also have that $z^2(1-z)^2$ is a maximal function in $K_{z^5}$ (since $C_{\ol z^5} z^2(1-z)^2 = z^2-1$ which is outer), so by Theorem \ref{thm:Cg_eigenfunctions} (ii) we have that $z(1-z^2)$ is an eigenvector of $C_{\ol z^5}$.
\end{ex}

%%%%%%%%%%%%%  Section 5  %%%%%%%%%%%%%
%%%%%         Applications        %%%%%

\section{Application to model spaces}
\label{sec:Applications}

We now apply the previous results to model spaces, a particularly well known class of Toeplitz kernels, thus recovering and extending, using a maximal function perspective, some known properties of those spaces.

Let $g=\ol \theta$, $\ker T_g = K_\theta$ where $\theta$ is an inner function, and let $C_\theta$ be the usual conjugation on $K_\theta$: $C_\theta f = \theta \ol z \ol f$, for every $f\in K_\theta$.
For any $\lambda\in\DD$, consider \kk\ and \kktil\ as in \eqref{eq:def_k_k^til}.
We have \kk, $\kktil\in K_\theta$ for all $\lambda\in\DD$, with $\kktil\in\Hinf$ and $\kk \in \GG\Hinf$;
\kktil\ is a maximal function in $K_\theta$ (see Example \ref{ex:ktil_is_max_func}).
Since $C_\theta \kktil = \kk$, an inner-outer factorization of \kktil, by Proposition \ref{prop:IO_fact_max_func_with_conj}, is
\begin{equation}
\label{eq:ktil_IO_fact}
\kktil = I O, \text{ where } I = \frac{\kktil}{\kk} \text{ is inner and } O=\kk.
\end{equation}

Thus \kktil\, for any $\lambda \in\DD$, is an example of a maximal function with square-rigid outer factor.
If $\theta(0) = 0$, then $\tilde k^\theta_0 = \ol z \theta$ is an inner function which is maximal in $K_\theta$, illustrating Proposition \ref{prop:inner_maximal_function}.

We can also obtain maximal functions in $K_\theta$ which are outer, using Corollary \ref{cor:outer_max_func}:
\begin{equation*}
O^M = \kk \left(\mu + \frac{\kktil}{\kk} \right) = \mu \kk + \kktil, \text{ with } \mu\in\CC, |\mu|=1.
\end{equation*}
These functions satisfy $C_\theta O^M = \ol \mu O^M$, they are eigenvectors of $C_\theta$ for the eigenvalue $\ol\mu$.
If we take $\mu=1$ then we get that $C_\theta (\kk + \kktil) = \kk + \kktil$ and we can write
\begin{equation}
K_\theta = \ker T_{\ol z \frac{\ol{\kk + \kktil}}{\kk + \kktil}} = \ker T_\frac{(\ol z + 1) (\ol{\kk + \kktil})}{(z+1) (\kk + \kktil)}
\end{equation}
(see Remark \ref{rem:kernel_symbol_outer_Hayashi}).
Note that, if $\dim\ker T_g >1$, an outer maximal function cannot be square-rigid.
For example $1+\alpha$ is an outer maximal function in $K_\theta = K_{z\alpha}$, where we assume that $\alpha\notin\CC$, but it is not square-rigid since
\begin{equation*}
\ker T_{\frac{1+\ol\alpha}{1+\alpha}} = \ker T_{\ol\alpha} = K_\alpha.
\end{equation*}

We now consider maximal function factorisations and their modified versions for the complex conjugate of inner functions.
From \eqref{eq:ktil_IO_fact} and the modified maximal function factorisation \eqref{eq:modified.max.func.fact} we have the representation
\begin{equation*}
\theta = B_{\lambda_1} \frac{\kktil}{\kk}\ \frac{\kk (1-\ol{\lambda_1} z)}{\ol\kk (1 - \lambda_1 \ol z)}, \quad \lambda,\lambda_1 \in\DD.
\end{equation*}

This can also be seen as a bounded factorisation of the form
\begin{equation*}
\theta = \frac{(\ol\kk)^{-1}}{1- \lambda_1 \ol z}\ \theta_{\lambda_1,\lambda}\ (\kk (1 - \ol{\lambda_1} z)),\ \text{ with } \theta_{\lambda_1,\lambda} = B_{\lambda_1} \frac{\kktil}{\kk}.
\end{equation*}
As a consequence, we have:

\begin{prop}
\label{prop:equiv.theta.and.theta.lambda1.lambda}
For all $\lambda,\lambda_1 \in\DD$, we have that $\theta \sim \theta_{\lambda_1,\lambda}$, $K_\theta \cong K_{\theta_{\lambda_1,\lambda}}$ and
\begin{equation}
\label{eq:equiv.theta.and.theta.lambda1.lambda_2}
K_\theta = \kk (1 - \ol{\lambda_1} z) K_{\theta_{\lambda_1,\lambda}}.
\end{equation}
\end{prop}

For $\lambda_1 = 0$, \eqref{eq:equiv.theta.and.theta.lambda1.lambda_2} corresponds to Proposition 25 in \cite{Cro94}.

Let us now see how the multiplier between the two model spaces in \eqref{eq:equiv.theta.and.theta.lambda1.lambda_2} relates the norm of an element of $K_\theta$ with its image in $K_{\theta_{\lambda_1,\lambda}}$.
Let $f_\theta$ be any element in $K_\theta$.
Using the equality
\begin{equation*}
\label{eq:fact_inverse_kk}
(\kk)^{-1}= (1- \ol\lambda z) \left(1 + \frac{\ol{\theta (\lambda)}}{1 - \ol{\theta(\lambda)}\theta} \theta \right),
\end{equation*}
we have 
\begin{align*}
\left|\left| 
	\frac{(\kk)^{-1}}{1- \ol\lambda_1 z} f_\theta 
\right|\right|^2
& = 
\left|\left|
	\frac{1- \ol\lambda z}{1- \ol\lambda_1 z} 
	\left( 
		1 + \frac{\ol{\theta (\lambda)}}{1 - \ol{\theta(\lambda)}\theta} \theta 
	\right)
	f_\theta  
\right|\right|^2 
=  \\
& = 
\left|\left| 
	\frac{1- \ol\lambda z}{1- \ol\lambda_1 z} f_\theta 
\right|\right|^2
+
|\theta(\lambda)|^2 
\left|\left| 
	\frac{(\kk)^{-1}}{1- \ol\lambda_1 z} f_\theta 
\right|\right|^2
+  \\
& \qquad +
2 \Re \left\langle 
	\frac{1- \ol\lambda z}{1- \ol\lambda_1 z}
	\frac{\ol{\theta (\lambda)}}{1 - \ol{\theta(\lambda)}\theta} 
	\theta f_\theta 
	, 
	\frac{1- \ol\lambda z}{1- \ol\lambda_1 z} f_\theta \right\rangle. 
\end{align*}
Choosing $\lambda_1=\lambda$, the inner product in the last term is zero.
Thus we get, in that case,
\begin{equation*}
( 1-|\theta(\lambda)|^2 ) \left|\left| \frac{1}{1 - \ol{\theta(\lambda)} \theta }  f_\theta \right|\right|^2  =  \left|\left| f_\theta \right|\right|^2,
\end{equation*}
meaning that the multiplier in \eqref{eq:equiv.theta.and.theta.lambda1.lambda_2}, for $\lambda_1=\lambda$, preserves the norm, up to a constant factor depending only on $\theta(\lambda)$.
Using the notation $\theta_\lambda := \theta_{\lambda,\lambda}$, we have then:

\begin{prop}
\label{prop:Crofoot_multiplier}
Let $\theta$ be an inner function and let
$\theta_\lambda = B_\lambda \kktil / \kk = \frac{\theta - \theta(\lambda)}{1 - \ol{\theta(\lambda)} \theta}$.
Then $\theta \sim \theta_\lambda$, $K_\theta \cong K_{\theta_\lambda}$ with 
$K_\theta =  M_{\theta,\lambda} K_{\theta_\lambda}$ where 
$M_{\theta,\lambda} = (1-\ol{\theta(\lambda)}\theta) / (\sqrt{1-|\theta(\lambda)|^2})$
is an isometric multiplier from $K_{\theta_\lambda}$ onto $K_\theta$.
\end{prop}

\begin{rem}
\label{rem:Crofoot_transform}
We recover here a result by Crofoot in \cite[Theorem 10(a)]{Cro94}.
The inverse of the operator defined from $K_{\theta_\lambda}$ onto $K_\theta$ by multiplication by $M_{\theta,\lambda}$, for each $\lambda\in\DD$, was called Crofoot's transformation, and denoted $J_\lambda$, by Sarason in \cite{Sa07}, where it was proven to define an isometry from $K_\theta$ onto $K_{\theta_\lambda}$ in a different way.
For $\lambda = 0$, in Proposition \ref{prop:Crofoot_multiplier}, we obtain Hayashi's representation of $K_\theta$, $K_\theta = M_{\theta,0} K_{\theta_0}$.
\end{rem}

\begin{rem}
Note that other factorisations of an inner function $\theta$ are possible, yielding different multipliers onto different model spaces, and one may be more convenient than the other, depending on the purpose.
For instance, if $\theta = B \in FBP$, using \eqref{eq:B-WH.fact}, one obtains that
\begin{equation}
\label{eq:mult_FBP_WH_fact}
K_B = r_+ K_{z^n}.
\end{equation}
where $n$ is the number of zeroes of $B$ and $r_+\in\GG\Hinf$ is a rational function.
This may be considered the simplest description of the model space $K_B$, since $K_{z^n}$ consists of polynomials.
It follows from \eqref{eq:mult_FBP_WH_fact} and Hayashi's representation in Theorem \ref{thm:Hayashi_rep} that every non-zero finite dimensional Toeplitz kernel can be represented in the form $\ker T_g = w K_{z^n}$ with $n= \dim \ker T_g$.
\end{rem}

\section*{Acknowledgments}
Research partially funded by Fundação para a Ciência e Tecnologia (FCT), Portugal, through grant No. UID/4459/2025.
Also, the work of the second author was supported by Fundação para a Ciência e Tecnologia (FCT), Portugal, through CAMGSD PhD fellowship UI/BD/153700/2022.

\bibliographystyle{apalike}
\bibliography{references}

@book {Dudu79,
    AUTHOR = {Duduchava, Roland},
     TITLE = {Integral equations in convolution with discontinuous
              presymbols, singular integral equations with fixed
              singularities, and their applications to some problems of
              mechanics},
 PUBLISHER = {BSB B. G. Teubner Verlagsgesellschaft, Leipzig},
      YEAR = {1979},
     PAGES = {172},
   MRCLASS = {45-02 (45E10)},
  MRNUMBER = {571890},
MRREVIEWER = {Martin\ Costabel},
}

@article {Ha85,
    AUTHOR = {Hayashi, Eric},
     TITLE = {The solution sets of extremal problems in {$H^1$}},
   JOURNAL = {Proc. Amer. Math. Soc.},
  FJOURNAL = {Proceedings of the American Mathematical Society},
    VOLUME = {93},
      YEAR = {1985},
    NUMBER = {4},
     PAGES = {690--696},
      ISSN = {0002-9939,1088-6826},
   MRCLASS = {30D55 (46E99)},
  MRNUMBER = {776204},
MRREVIEWER = {J.\ A.\ Cima},
       DOI = {10.2307/2045546},
       URL = {https://doi.org/10.2307/2045546},
}

@book {MiPro86,
    AUTHOR = {Mikhlin, Solomon G. and Pr\"ossdorf, Siegfried},
     TITLE = {Singular integral operators},
 PUBLISHER = {Springer-Verlag, Berlin},
      YEAR = {1986},
     PAGES = {528},
      ISBN = {3-540-15967-3},
   MRCLASS = {47G05 (45-02 58G10)},
  MRNUMBER = {881386},
MRREVIEWER = {C.\ Carton-Lebrun},
       DOI = {10.1007/978-3-642-61631-0},
       URL = {https://doi.org/10.1007/978-3-642-61631-0},
}

@article {Ha86,
    AUTHOR = {Hayashi, Eric},
     TITLE = {The kernel of a {T}oeplitz operator},
   JOURNAL = {Integral Equations Operator Theory},
  FJOURNAL = {Integral Equations and Operator Theory},
    VOLUME = {9},
      YEAR = {1986},
    NUMBER = {4},
     PAGES = {588--591},
      ISSN = {0378-620X,1420-8989},
   MRCLASS = {47B35},
  MRNUMBER = {853630},
MRREVIEWER = {Jeffrey\ R.\ Butz},
       DOI = {10.1007/BF01204630},
       URL = {https://doi.org/10.1007/BF01204630},
}

@book {LitSpit87,
    AUTHOR = {Litvinchuk, Georgii S. and Spitkovskii, Ilia M.},
     TITLE = {Factorization of measurable matrix functions},
    SERIES = {Operator Theory: Advances and Applications},
    VOLUME = {25},
 PUBLISHER = {Birkh\"auser Verlag, Basel},
      YEAR = {1987},
     PAGES = {372},
      ISBN = {3-7643-1883-X},
   MRCLASS = {47A68 (30E25 45E10 47A45 47B38)},
  MRNUMBER = {1015716},
       DOI = {10.1007/978-3-0348-6266-0},
       URL = {https://doi.org/10.1007/978-3-0348-6266-0},
}

@article {Hitt88,
    AUTHOR = {Hitt, D.},
     TITLE = {Invariant subspaces of {$H^2$} of an annulus},
   JOURNAL = {Pacific J. Math.},
  FJOURNAL = {Pacific Journal of Mathematics},
    VOLUME = {134},
      YEAR = {1988},
    NUMBER = {1},
     PAGES = {101--120},
      ISSN = {0030-8730,1945-5844},
   MRCLASS = {46E20 (46J15 47A15 47B38)},
  MRNUMBER = {953502},
MRREVIEWER = {John\ N.\ McDonald},
       URL = {http://projecteuclid.org/euclid.pjm/1102689368},
}

@article {Ha90,
    AUTHOR = {Hayashi, Eric},
     TITLE = {Classification of nearly invariant subspaces of the backward shift},
   JOURNAL = {Proc. Amer. Math. Soc.},
  FJOURNAL = {Proceedings of the American Mathematical Society},
    VOLUME = {110},
      YEAR = {1990},
    NUMBER = {2},
     PAGES = {441--448},
      ISSN = {0002-9939,1088-6826},
   MRCLASS = {47A15 (30H05 47B35 47B38)},
  MRNUMBER = {1019277},
MRREVIEWER = {James\ Rovnyak},
       DOI = {10.2307/2048087},
       URL = {https://doi.org/10.2307/2048087},
}

@book {GoKru92,
    AUTHOR = {Gohberg, Israel and Krupnik, Naum},
     TITLE = {One-dimensional linear singular integral equations. {I \& II}},
    SERIES = {Operator Theory: Advances and Applications},
    VOLUME = {53 \& 54},
 PUBLISHER = {Birkh\"auser Verlag, Basel},
      YEAR = {1992},
     PAGES = {266},
      ISBN = {3-7643-2584-4},
   MRCLASS = {47G10 (45E10 47A53 47B35 47N20)},
  MRNUMBER = {1138208},
MRREVIEWER = {A.\ B\"ottcher},
}

@article {Cro94,
    AUTHOR = {Crofoot, R. Bruce},
     TITLE = {Multipliers between invariant subspaces of the backward shift},
   JOURNAL = {Pacific J. Math.},
  FJOURNAL = {Pacific Journal of Mathematics},
    VOLUME = {166},
      YEAR = {1994},
    NUMBER = {2},
     PAGES = {225--246},
      ISSN = {0030-8730,1945-5844},
   MRCLASS = {47A15 (30H05 46E22 46J15 47B38)},
  MRNUMBER = {1313454},
MRREVIEWER = {James\ Rovnyak},
       URL = {http://projecteuclid.org/euclid.pjm/1102621137},
}

@incollection {Sa94,
    AUTHOR = {Sarason, Donald},
     TITLE = {Kernels of {T}oeplitz operators},
 BOOKTITLE = {Toeplitz operators and related topics ({S}anta {C}ruz, {CA}, 1992)},
    SERIES = {Oper. Theory Adv. Appl.},
    VOLUME = {71},
     PAGES = {153--164},
 PUBLISHER = {Birkh\"auser, Basel},
      YEAR = {1994},
      ISBN = {3-7643-5068-7},
   MRCLASS = {47B35 (30D55 46J15 47A15)},
  MRNUMBER = {1300218},
MRREVIEWER = {Takahiko\ Nakazi},
}

@book {Ni02_vol1,
    AUTHOR = {Nikolski, Nikolai K.},
     TITLE = {Operators, functions, and systems: an easy reading. {V}ol. 1},
    SERIES = {Mathematical Surveys and Monographs},
    VOLUME = {92},
 PUBLISHER = {Amer. Math. Soc., Providence, RI},
      YEAR = {2002},
     PAGES = {xiv+461},
      ISBN = {0-8218-1083-9},
   MRCLASS = {47-02 (30D55 30E05 30H05 46Exx 47N70 93B28)},
  MRNUMBER = {1864396},
MRREVIEWER = {Harry\ Dym},
}

@incollection {MaPol05,
    AUTHOR = {Makarov, N. and Poltoratski, A.},
     TITLE = {Meromorphic inner functions, {T}oeplitz kernels and the
              uncertainty principle},
 BOOKTITLE = {Perspectives in analysis},
    SERIES = {Math. Phys. Stud.},
    VOLUME = {27},
     PAGES = {185--252},
 PUBLISHER = {Springer, Berlin},
      YEAR = {2005},
      ISBN = {978-3-540-30432-6; 3-540-30432-0},
   MRCLASS = {47B35 (30C40 46E22 47B32 47E05)},
  MRNUMBER = {2215727},
MRREVIEWER = {Anton\ Baranov},
       DOI = {10.1007/3-540-30434-7\_10},
       URL = {https://doi.org/10.1007/3-540-30434-7_10},
}

@article {Sa07,
    AUTHOR = {Sarason, Donald},
     TITLE = {Algebraic properties of truncated {T}oeplitz operators},
   JOURNAL = {Oper. Matrices},
  FJOURNAL = {Operators and Matrices},
    VOLUME = {1},
      YEAR = {2007},
    NUMBER = {4},
     PAGES = {491--526},
      ISSN = {1846-3886,1848-9974},
   MRCLASS = {47B35 (47B32)},
  MRNUMBER = {2363975},
MRREVIEWER = {William\ Thomas\ Ross},
       DOI = {10.7153/oam-01-29},
       URL = {https://doi.org/10.7153/oam-01-29},
}

@book {ChaPar11,
    AUTHOR = {Chalendar, Isabelle and Partington, Jonathan R.},
     TITLE = {Modern approaches to the invariant-subspace problem},
    SERIES = {Cambridge Tracts in Mathematics},
    VOLUME = {188},
 PUBLISHER = {Cambridge University Press, Cambridge},
      YEAR = {2011},
     PAGES = {xii+285},
      ISBN = {978-1-107-01051-2},
   MRCLASS = {47-02 (47A15 47A60)},
  MRNUMBER = {2841051},
MRREVIEWER = {Sergei\ M.\ Shimorin},
       DOI = {10.1017/CBO9780511862434},
       URL = {https://doi.org/10.1017/CBO9780511862434},
}

@article {CP14,
    AUTHOR = {C\^amara, M. Cristina and Partington, Jonathan R.},
     TITLE = {Near invariance and kernels of {T}oeplitz operators},
   JOURNAL = {J. Anal. Math.},
  FJOURNAL = {Journal d'Analyse Math\'ematique},
    VOLUME = {124},
      YEAR = {2014},
     PAGES = {235--260},
      ISSN = {0021-7670,1565-8538},
   MRCLASS = {47B35 (30H10)},
  MRNUMBER = {3286053},
MRREVIEWER = {Akio\ Arimoto},
       DOI = {10.1007/s11854-014-0031-8},
       URL = {https://doi.org/10.1007/s11854-014-0031-8},
}

@article {CMP16,
    AUTHOR = {C\^amara, M. C. and Malheiro, M. T. and Partington, J. R.},
     TITLE = {Model spaces and {T}oeplitz kernels in reflexive {H}ardy
              space},
   JOURNAL = {Oper. Matrices},
  FJOURNAL = {Operators and Matrices},
    VOLUME = {10},
      YEAR = {2016},
    NUMBER = {1},
     PAGES = {127--148},
      ISSN = {1846-3886,1848-9974},
   MRCLASS = {47B35 (30H10)},
  MRNUMBER = {3460059},
MRREVIEWER = {William\ Thomas\ Ross},
       DOI = {10.7153/oam-10-09},
       URL = {https://doi.org/10.7153/oam-10-09},
}

@article {CP16_finite,
    AUTHOR = {C\^amara, M. C. and Partington, J. R.},
     TITLE = {Finite-dimensional {T}oeplitz kernels and nearly-invariant
              subspaces},
   JOURNAL = {J. Operator Theory},
  FJOURNAL = {Journal of Operator Theory},
    VOLUME = {75},
      YEAR = {2016},
    NUMBER = {1},
     PAGES = {75--90},
      ISSN = {0379-4024,1841-7744},
   MRCLASS = {47B35 (30E25 30H10)},
  MRNUMBER = {3474097},
       DOI = {10.7900/jot.2014oct29.2067},
       URL = {https://doi.org/10.7900/jot.2014oct29.2067},
}

@book {GarMashRoss16,
    AUTHOR = {Garcia, Stephan Ramon and Mashreghi, Javad and Ross, William
              T.},
     TITLE = {Introduction to model spaces and their operators},
    SERIES = {Cambridge Studies in Advanced Mathematics},
    VOLUME = {148},
 PUBLISHER = {Cambridge University Press, Cambridge},
      YEAR = {2016},
     PAGES = {xv+322},
      ISBN = {978-1-107-10874-5},
   MRCLASS = {30H10 (30H05 30J05 42B30 47B37)},
  MRNUMBER = {3526203},
MRREVIEWER = {Steven\ M.\ Deckelman},
       DOI = {10.1017/CBO9781316258231},
       URL = {https://doi.org/10.1017/CBO9781316258231},
}

@article {Ca17,
    AUTHOR = {C\^amara, M. Cristina},
     TITLE = {Toeplitz operators and {W}iener-{H}opf factorisation: an
              introduction},
   JOURNAL = {Concr. Oper.},
  FJOURNAL = {Concrete Operators},
    VOLUME = {4},
      YEAR = {2017},
    NUMBER = {1},
     PAGES = {130--145},
      ISSN = {2299-3282},
   MRCLASS = {47B35 (45E10 47A68)},
  MRNUMBER = {3724468},
       DOI = {10.1515/conop-2017-0010},
       URL = {https://doi.org/10.1515/conop-2017-0010},
}

@article {CP18_multipliers,
    AUTHOR = {C\^amara, M. Cristina and Partington, Jonathan R.},
     TITLE = {Multipliers and equivalences between {T}oeplitz kernels},
   JOURNAL = {J. Math. Anal. Appl.},
  FJOURNAL = {Journal of Mathematical Analysis and Applications},
    VOLUME = {465},
      YEAR = {2018},
    NUMBER = {1},
     PAGES = {557--570},
      ISSN = {0022-247X,1096-0813},
   MRCLASS = {47B35 (42A45)},
  MRNUMBER = {3806717},
MRREVIEWER = {Christoph\ Kriegler},
       DOI = {10.1016/j.jmaa.2018.05.023},
       URL = {https://doi.org/10.1016/j.jmaa.2018.05.023},
}

@article {CMP21,
    AUTHOR = {C\^amara, M. C. and Malheiro, M. T. and Partington, J. R.},
     TITLE = {Kernels of unbounded {T}oeplitz operators and factorization of
              symbols},
   JOURNAL = {Results Math.},
  FJOURNAL = {Results in Mathematics},
    VOLUME = {76},
      YEAR = {2021},
    NUMBER = {10},
      ISSN = {1422-6383,1420-9012},
   MRCLASS = {47B35 (45E10 47A68)},
  MRNUMBER = {4193428},
MRREVIEWER = {Ilie\ Valu\c sescu},
       DOI = {10.1007/s00025-020-01323-z},
       URL = {https://doi.org/10.1007/s00025-020-01323-z},
}

@article {KiAbMiRo21,
    AUTHOR = {Kisil, Anastasia V. and Abrahams, I. David and Mishuris,
              Gennady and Rogosin, Sergei V.},
     TITLE = {The {W}iener-{H}opf technique, its generalizations and
              applications: constructive and approximate methods},
   JOURNAL = {Proc. R. Soc. A.},
  FJOURNAL = {Proceedings of the Royal Society A},
    VOLUME = {477},
      YEAR = {2021},
    NUMBER = {2254},
     PAGES = {20210533},
      ISSN = {1364-5021,1471-2946},
   MRCLASS = {45A05 (30E25)},
  MRNUMBER = {4340437},
       DOI = {10.1098/rspa.2021.0533},
       URL = {https://doi.org/10.1098/rspa.2021.0533},
}

@article {DyPlaPtak22,
    AUTHOR = {Dymek, Piotr and P{\l}aneta, Artur and Ptak, Marek},
     TITLE = {Conjugations preserving {T}oeplitz kernels},
   JOURNAL = {Integral Equations Operator Theory},
  FJOURNAL = {Integral Equations and Operator Theory},
    VOLUME = {94},
      YEAR = {2022},
    NUMBER = {39},
      ISSN = {0378-620X,1420-8989},
   MRCLASS = {47A45 (47B37)},
  MRNUMBER = {4502709},
MRREVIEWER = {Sanne\ ter Horst},
       DOI = {10.1007/s00020-022-02714-3},
       URL = {https://doi.org/10.1007/s00020-022-02714-3},
}

@article {KiMi23,
    AUTHOR = {Kisil, Anastasia and Mishuris, Gennady},
     TITLE = {Special issue: advances in {W}iener-{H}opf type techniques:
              theory and applications},
   JOURNAL = {Proc. R. Soc. A.},
  FJOURNAL = {Proceedings of the Royal Society A},
    VOLUME = {479},
      YEAR = {2023},
    NUMBER = {2273},
     PAGES = {20230210},
      ISSN = {1364-5021,1471-2946},
   MRCLASS = {35-06},
  MRNUMBER = {4606460},
       DOI = {10.1098/rspa.2023.0210},
       URL = {https://doi.org/10.1098/rspa.2023.0210},
}

@article {CP24,
    AUTHOR = {C\^amara, M. Cristina and Partington, Jonathan R.},
     TITLE = {Paired kernels and their applications},
   JOURNAL = {Results Math.},
  FJOURNAL = {Results in Mathematics},
    VOLUME = {79},
      YEAR = {2024},
    NUMBER = {120},
      ISSN = {1422-6383,1420-9012},
   MRCLASS = {47B35 (30H10 47B38)},
  MRNUMBER = {4718705},
MRREVIEWER = {Chunxu\ Xu},
       DOI = {10.1007/s00025-024-02146-y},
       URL = {https://doi.org/10.1007/s00025-024-02146-y},
}

@article {CamaraCardoso24,
    AUTHOR = {C\^amara, M. Cristina and Cardoso, Gabriel Lopes},
     TITLE = {Riemann-{H}ilbert problems, {T}oeplitz operators and
              ergosurfaces},
   JOURNAL = {J. High Energy Phys.},
  FJOURNAL = {Journal of High Energy Physics},
      YEAR = {2024},
    VOLUME = {6},
    NUMBER = {27},
      ISSN = {1126-6708,1029-8479},
   MRCLASS = {35Q15 (83C20 83E15)},
  MRNUMBER = {4763004},
MRREVIEWER = {Marcelo\ Santos\ Guimaraes},
       DOI = {10.1007/jhep06(2024)027},
       URL = {https://doi.org/10.1007/jhep06(2024)027},
}

@incollection {CamaraCardoso25,
    AUTHOR = {C\^amara, M. Cristina and Cardoso, Gabriel Lopes},
     TITLE = {Factorisation of symmetric matrices and applications in
              gravitational theories},
 BOOKTITLE = {Achievements and challenges in the field of convolution
              operators},
    SERIES = {Oper. Theory Adv. Appl.},
    VOLUME = {306},
     PAGES = {157--181},
 PUBLISHER = {Birkh\"auser/Springer, Cham},
      YEAR = {2025},
      ISBN = {978-3-031-80485-4; 978-3-031-80486-1},
   MRCLASS = {47A56 (30H99 45E10)},
  MRNUMBER = {4888864},
       DOI = {10.1007/978-3-031-80486-1\_6},
       URL = {https://doi.org/10.1007/978-3-031-80486-1_6},
}

\end{document}